\documentclass{amsart}

\usepackage{multirow,hhline}
\usepackage{amssymb,amsmath,amsfonts,amsthm,url,bm,stmaryrd}
\usepackage{a4wide}
\usepackage[mathscr]{eucal}
\usepackage{bm}

\usepackage{pgf}
\usepackage{tikz}

\usetikzlibrary{calc,positioning,decorations.pathmorphing,matrix,arrows,shapes.geometric}

\makeatletter
\tikzset{gloop/.style =  {to path={
  \pgfextra{}
  [looseness=3,min distance=6mm,-stealth,thick, dotted, shorten >=1pt]
  \tikz@to@curve@path}
  }}  
\makeatletter 

\makeatletter
\tikzset{ggloop/.style =  {to path={
  \pgfextra{}
  [looseness=3,min distance=6mm,-stealth,thick,  dashed, shorten >=1pt]
  \tikz@to@curve@path}
  }}  
\makeatletter 

\makeatletter
\tikzset{kloop/.style =  {to path={
  \pgfextra{}
  [looseness=5,min distance=6mm,-stealth,shorten >=1pt]
  \tikz@to@curve@path},font=\sffamily\small
  }}  
\makeatletter

\theoremstyle{plain}
\newtheorem{theorem}{Theorem}[section]
\newtheorem{lemma}[theorem]{Lemma}
\newtheorem{proposition}[theorem]{Proposition}
\newtheorem{corollary}[theorem]{Corollary}

\theoremstyle{definition}
\newtheorem{ex}[theorem]{Example}

\newenvironment{newlist}
   {\begin{list}{}{\setlength{\labelsep}{0.25cm}
                   \setlength{\labelwidth}{0.65cm}
                      \setlength{\leftmargin}{0.9cm}}}
   {\end{list}}
\newenvironment{claimlist}
   {\begin{list}{}{\setlength{\labelsep}{0.25cm}
                   \setlength{\labelwidth}{0.65cm}
                      \setlength{\leftmargin}{0.4cm}}}
   {\end{list}}

\newcommand{\newapprox}{\simeq}
\newcommand{\newcover}{\raise-.18ex\hbox{\Large$ \Yleft$}} 
\renewcommand{\leq}{\leqslant}
\renewcommand{\geq}{\geqslant} 
\renewcommand{\phi}{\varphi}

\newcommand{\cat}[1]{\boldsymbol{\mathscr{#1}}}
\newcommand{\str}[1]{\mathbf{#1}}
\newcommand{\fnt}[1]{\mathsf{#1}}
\newcommand{\ope}[1]{\mathbb{#1}}
\newcommand{\defn}[1]{{\emph{#1}}}

\newcommand{\Tp}{\mathscr{T}}

\newcommand{\ISP}{\ope{ISP}}
\newcommand{\HSP}{\ope{HSP}}

\DeclareMathOperator{\graph}{graph}
\DeclareMathOperator{\ran}{ran}
\DeclareMathOperator{\dom}{dom}
\DeclareMathOperator{\id}{id}
  
\DeclareMathOperator{\End}{End}
\newcommand{\PE}{\End_{\text{p}}}    
\newcommand{\CnT}{\twiddle{\Cn}}
\newcommand{\CXn}{\CX_n}

\newcommand{\Gn}{\cat{G}_n}
\newcommand{\G}{\cat{G}}
\newcommand{\Cn}{\str{C}_n}
\newcommand{\CFn}{\cat{F}_n}
\newcommand{\CA}{\cat A}
\newcommand{\CX}{\cat X}
\newcommand{\CP}{\cat P}
\newcommand{\CCD}{\cat D}
\newcommand{\X}{\str X}
\newcommand{\XX}{{\scriptscriptstyle \str X}}
\newcommand{\YY}{{\scriptscriptstyle \str Y}}
\newcommand{\ZZ}{{\scriptscriptstyle \str Z}}
\newcommand{\Y}{\str Y}
\newcommand{\Z}{\str Z}
\newcommand{\M}{\str M}
\newcommand{\A}{\str A}
\newcommand{\B}{\str B}
\newcommand{\C}{\str C}
\newcommand{\Free}{\mathbf{F}}
\newcommand{\D}{\fnt{D}}
\newcommand{\E}{\fnt{E}}
\newcommand{\Lalg}{\str L}
\newcommand{\two}{\boldsymbol 2}
\newcommand{\twoT}{\twiddle{\boldsymbol 2}}

\newcommand{\w}{\omega}
\newcommand{\restrict}[1]{{\upharpoonright}_{#1}}

\newcommand{\powerset}[1]{{
 \raise.5ex\hbox{\large $\wp$}#1}}
\newcommand{\powersetn}{
\raise.4ex\hbox{\large $\wp$}_n}

\newcommand{\twiddle}[1]{{\smash{\underset{\raise.375ex\hbox{$\smash\sim$}}
       {#1}}\vphantom{\underline{#1}}}} 
 \DeclareMathOperator{\IScP}{{\ope{IS} _{\mathrm{c}}%
 \ope{P}^+}}
\DeclareMathOperator{\Max}{max}  
\newcommand{\MT}{\twiddle{\M}}

\newcommand{\lincol}{black}
\newcommand{\linth}{thick}
\newcommand{\po}[2][\pocol]{\filldraw[#1](#2) circle (2 pt);}
\newcommand{\epo}[1]{\filldraw[fill=white,\linth](#1) circle (2pt);}
\newcommand{\li}[1]{\draw[\linth,\lincol] #1;}
\newcommand{\dotli}[1]{\draw[dotted,\linth,\lincol] #1;}

\newcommand{\ggarr}[1]{\draw[-stealth,thick,  dashed, shorten <= 1.3pt,
shorten >=1.3pt]  #1;}
\newcommand{\garr}[1]{\draw[-stealth,thick, dotted, shorten <= 1.3pt,
shorten >=1.3pt]  #1;}
\newcommand{\harr}[1]{\draw[-stealth, 
shorten >=1.3pt]  #1;}
\newcommand{\karr}[1]{\draw[-stealth, 
shorten >=1.3pt]  #1;}
\newcommand{\earr}[1]{\draw[-stealth, thick, dotted,shorten <= 1.3pt,
shorten >=1.3pt]  #1;}
\newcommand{\eearr}[1]{\draw[-stealth,thick,  dashed, shorten <= 1.3pt,
shorten >=1.3pt]  #1;}

\tikzset{every.loop/.style={
looseness=30}}

\begin{document}
\title[]{G\"odel algebras: interactive dualities and their applications}

\dedicatory{Dedicated to Brian Davey in celebration of his 65th birthday}	
\thanks{The first author was supported by a Marie Curie Intra European Fellowship within the 7th European Community Framework Program (ref. 299401-FP7-PEOPLE-2011-IEF). }

\author{L.\,M.\,Cabrer}
\email{l.cabrer@disia.unifi.it}
\address[lmc]{Dipartimento di Statistica, Informatica, Applicazioni, Universit\`a degli Studi di Firenze, 59 Viale Morgani, 50134, Italy}
\author{H.\,A. Priestley}
\email{hap@maths.ox.ac.uk}
\address{Mathematical Institute, University of Oxford, Radcliffe Observatory Quarter,
Oxford OX2 6GG, United Kingdom}
\begin{abstract}
We present a  technique for deriving certain new natural dualities for any variety of  algebras generated by a finite Heyting chain. The dualities we construct are tailored to admit a transparent translation to the more pictorial Priestley/Esakia duality and back again. This enables us to combine  the two approaches and so to capitalise on  the virtues of both, in particular the categorical good behaviour of a natural duality: we thereby demonstrate the fullness, or not, of  each of our dualities; we obtain new results on amalgamation; and we also provide a simple treatment of coproducts.
\end{abstract}

\maketitle

\section{Introduction}  \label{sec:intro}

This paper focuses on the classes of algebras  $\Gn = \ISP(\Cn)$, where $\Cn$ is the $n$-element Heyting  chain, and on natural dualities for them.  As we recall below, the 
classes $\Gn$  are important both within and beyond duality theory.

It is highly 
appropriate that this topic should feature in the  Special Issue of Algebra 
Universalis in honour of Brian Davey: through his work, the  
 varieties $\Gn$ have been influential 
as a  example within natural duality theory  for nearly 40 years, beginning with 
 his trail-blazing 1976 paper \cite{D76}.  The advances to which study of these
varieties have made a contribution are well chronicled, 
and fully referenced,  
by Davey and Talukder  \cite[Section~1]{DT05}.  Here we add a further chapter to the saga.   
Although the emphasis will be squarely on the varieties $\Gn$, 
which  have very special features which work to our advantage, 
glimpses also open up of  future avenues of investigation  with 
wider scope.

Many varieties which provide illuminating test case examples for duality theory 
are also of relevance to logic. 
In particular any variety 
$\CA$ 
of Heyting algebras is the algebraic counterpart of an 
intermediate logic $\mathcal{L}_{\CA}$ (an axiomatic extension of 
intuitionistic propositional logic, IPC). 
The varieties $\Gn$
are the
proper 
subvarieties of the variety  $\G$ of
G\"odel algebras  (also known as G\"odel/Dummett algebras, $\Lalg$-algebras, or pre-linear Heyting algebras),  {\it viz.}~ the Heyting algebras that satisfy the pre-linearity equation $(a\to b)\lor (b\to a)\approx \top$. 
They are the equivalent
algebraic  semantics
of the  
 G\"odel/Dummett extension of IPC  obtained by adding the linearity axiom 
$(x\to y)\vee (y\to x)$   (see
 \cite{Go,Du} and  also \cite{Pl} for further historical references).
In general, 
if a  quasivariety $\CA$ is the equivalent algebraic 
semantics for a sentential logic $\mathcal{L}_{\CA}$,
then there are connections between 
algebraic properties in $\CA$ and logical properties of 
$\mathcal{L}_{\CA}$. 
For example, different 
amalgamation properties in~$\CA$ correspond to different  interpolation 
properties in $\mathcal{L}_{\CA}$ (see for example \cite{CzPi}).
 This leads us
to study amalgamation in G\"odel algebras, as reported below.

Duality theory for Heyting algebras is intimately connected with 
Kripke semantics  for IPC.   Paralleling  these relational semantics
there is  the  topological duality  we shall refer to as Priestley/Esakia duality \cite{Esa}.
This specialises Priestley duality for the category
$\CCD$ of bounded distributive lattices to Heyting algebras
and provides a primary tool for the study of such algebras.
  It is  well known that 
a Heyting algebra belongs to $\Gn$ if  and only if
the order of 
its associated Esakia space is a forest of depth at most $n-2$.
Further details are given in Section~\ref{sec:Translation}.  
An overarching objective of this paper may be seen as the development of 
a closer tie-up than hitherto available between natural dualities for
the varieties $\Gn$ and the Priestley/Esakia duality.

 We have aimed to make our results  accessible to those interested in G\"odel algebras {\it per se\/} and  in the applications to amalgamation and coproducts we give in Section~6.  Nonetheless,  we cannot make our 
account  fully  self-contained and shall
refer to Clark and Davey's text \cite{CD98} for background on 
the fundamentals of natural duality theory.
The following  summary of certain key events  
is  directed at those already 
conversant with this theory.  
It will enable us to set in context for such readers what we achieve in this paper. 

It was established in \cite{D76} 
that each variety $\Gn$ is endodualisable (so that the alter ego 
$(C_n; \End \Cn, \Tp)$ yields a duality); in particular $\Gn$ is dually 
equivalent to a category of Boolean spaces acted on  by a monoid of continuous maps. 
 While endodualisability of $\Gn$ 
ensures that the alter ego is of an amenable type, 
the endomorphism monoid 
$\End\Cn$ grows exponentially as $n$ increases. 
This issue was addressed by Davey and Talukder \cite[Section~2]{DT05}, with consideration of optimality within the realm of dualities for $\Gn$ based on endomorphisms.  But there is another approach worthy of consideration. 
 It was to obtain more tractable dualities than those supplied by the NU Duality Theorem,  as it applies in particular
to distributive lattice-based algebras, that Davey and Werner \cite{DW} devised
their  `piggyback method'.  This leads to  a reasonably economical choice of alter ego, which in the case of $\Gn$  contains (the graphs of) the members of a 
family containing both endomorphisms and partial endomorphisms.   
While natural duality theory  was set up from the outset to encompass alter egos containing partial operations, the perception has always been that total structures
are to be preferred, whenever possible.  However, with a recent 
spurt of progress in understanding alter egos, the role  of partial operations
is steadily becoming less mysterious. 
Moreover, adding partial operations to upgrade a duality to a strong, and hence full, duality is  a standard technique for creating full dualities for varieties of 
lattice-based algebras.  Hence including certain partial endomorphisms in an alter ego 
for a variety $\Gn$ may be desirable on grounds of economy, and also 
the best option for  achieving fullness when a full duality is wanted.

The piggyback 
 strategy leans heavily on  Priestley duality for  $\CCD$ 
 and it is natural to go one step further and to seek to relate 
the piggyback natural duality for a $\CCD$-based quasivariety $\CA$
to 
Priestley duality applied to $\fnt{U}(\A)$,  where $\fnt{U} \colon \CA \to \CCD$ is the obvious forgetful functor.   
Indeed, the germ of the idea for the piggyback method is already present 
in Davey's proof of his duality for~$\Gn$ \cite[Theorem~2.4]{D76}.  
Esakia, in his review 
MR0412063 (54 \#192) of \cite{D76}, observed that it would be 
interesting to compare
Davey's duality for $\Gn$ with the duality in Esakia's own  1974 paper \cite{Esa}.   
Davey had already taken the first step here:  within the proof 
of \cite[Theorem~2.4]{D76}  he 
shows how to pass from the natural dual of an algebra 
$\A\in\Gn$ to the Priestley/Esakia dual of its $\CCD$-reduct.  
He does this by identifying the latter  
with a Priestley space obtained as the quotient 
of the natural dual space of $\A$, where the equivalence relation and the ordering on the quotient are determined by the action of the endomorphisms of~$\Cn$.
The present authors in \cite{CPcop} discussed 
an analogous process in the context of a piggyback  duality for 
any finitely generated quasivariety of $\CCD$-based algebras. 
This translation 
was developed to facilitate an analysis of coproducts in finitely generated 
$\CCD$-based quasivarieties.  
Here we take these ideas  further. 
For the dualities we present for each class $\Gn$, we are able to set up a simple
two-way translation, in a functorial way, between the dual categories involved.  
Simplicity stems from our choices of alter ego; these result 
a quotienting process which is particularly easy
to visualise.  
But what is much more significant for applications is the bi-directional nature of our translation.  
We shall refer to the dualities for which
such a translation is available as \defn{interactive}.
With an interactive duality to hand we  have,   
in a strong sense,
a `best of both worlds'
scenario: 
we can harness both  the categorical virtues of a natural duality 
and  the merits of Priestley duality,
specifically its pictorial character and the fact that its is a strong natural duality.

We now outline  the structure of this paper and indicate, in somewhat more detail
than above, what we achieve.
 In Section~2  we conduct a  detailed analysis of the elements in the hom-sets
$\Gn(\A,\Cn)$ for $\A \in \Gn$.
This includes 
investigation  of the images of these maps in 
$\CCD(\fnt{U}(\A),\two)$ under  $\Phi_{\w} \colon x \mapsto \w \circ x$, 
where  $\w \colon \fnt{U}(\Cn) \to \{0,1\}$ is the $\CCD$-homomorphism which sends  the top element of $\Cn$ to~$1$ and all other elements to~$0$. 
Some of our results are well known (surjectivity of $\Phi_\w$, 
for example, is a key component in the validation of Davey and 
Werner's piggyback duality) and
certain ingredients 
are common to our treatment of endomorphisms and that of Davey \cite{D76}, but 
other results  are new.
 In particular Lemmas
\ref{lem:isoupchain}--\ref{lem:endo-range} 
go beyond what appears in the existing literature, and suggest 
that particular  partial endomorphisms and endomorphisms of~$\Cn$ may be
good  candidates for inclusion in an alter ego for $\Cn$  tailored to  a smooth
translation between the associated natural duality for $\Gn$ and Priestley duality.
In Section~\ref{sec:natdual} we combine duality theory's Test Algebra Lemma with an adaptation of  the proof of the traditional Piggyback Duality Theorem 
to obtain a family of new dualities for $\Gn$   (for $n \geq 4$).     
Each 
includes in the alter ego $n-3$ partial endomorphisms 
and one of $n-2$ endomorphisms.  
In Section~\ref{sec:Translation} we confirm that our new dualities are indeed tailor-made for two-way translation. 
 We adapt 
the strategy used in \cite[Theorem~2.4]{CPcop} to pass, with the aid of 
$\w$, from the natural dual of an algebra to the Priestley dual of its reduct. 
More significantly, the results in Section~\ref{sec:prelim} allow us also to go back again
(Theorem~\ref{thm:going back}).

   We demonstrate  the 
 power of our interactive dualities  first in
Section~\ref{Sec:Fullness}.  We are able to show 
that for each $n$ only one of our $2^{n-3}\times (n-2)$
 dualities is full. 
(We recall that 
Davey \cite{D76} showed that $\End\Cn$ does not yield a full duality on $\Gn$ when $n \geq 4$ and it is known 
that the entire monoid 
$\PE\Cn$ of partial and total  
endomorphisms of $\Cn$ 
yields a full duality.)
Section~\ref{sec:coprod} is devoted to two different applications: to amalgamation and to coproducts.
It follows from results of Maksimova 
\cite{Maks} that $\Gn$ fails to satisfy the amalgamation property 
if $n \geq 4$, so that not every 
$V$-formation admits amalgamation in $\G_n$. 
We use our two-way translation to determine which
$V$-formations 
do admit amalgamation in $\G_n$. 
This result is based on the categorical properties of natural dualities 
and the fact that Priestley duality maps injective homomorphisms 
to surjective continuous maps.
Finally we extend our work on coproducts  \cite{CPcop}
by adding G\"odel algebras to the catalogue of examples 
provided there. 
  We employ our two-way translation to give a procedure for 
describing coproducts in $\Gn$ and, in certain cases,  in $\G $ too.
Our
method
provides a simple alternative to the procedure
presented by  D'Antona and Marra \cite{DM} in the case of finite 
G\"odel algebras (they 
employ solely 
Priestley/Esakia duality) and by Davey \cite[Section~5]{D76} for 
algebras in $\Gn$ (he uses only his 
natural duality for $\Gn$).

We elect 
to formulate our principal results about dualities for $\Gn$
 under the assumption $n\geq 4$
since to encompass $n=2,3$ would
complicate the statements and contribute little that is new.  However, {\it mutatis mutandis}, the special cases
can be fitted into our general scheme, and we make brief comments as we proceed 
to confirm this.

  \section{G\"odel algebras} 
\label{sec:prelim}

An algebra $\A=(A; \wedge,\vee,\to,\bot,\top)$ is  a \defn{Heyting algebra} if 
$(A;\wedge,\vee,\bot,\top)$ is a bounded distributive lattice and 
$a\wedge b\leq c$ if and only if $a\leq b\to c$.  
A basic reference
for the algebraic properties of Heyting algebras is \cite[Chapter~IX]{BaDw}.
We shall denote the variety of Heyting algebras by  $\cat H$.  
We make this, and likewise any other class of algebras with which 
we work, into a category by taking as morphisms all 
homomorphisms.

It will be important 
that any Heyting algebra 
has a reduct 
 in the variety~$\CCD$ of bounded 
distributive lattices.  
More precisely, we have a forgetful functor $\fnt{U}\colon 
\cat H \to \CCD$ 
that on objects sends  
$\A = (A; \land,\lor,\to ,\bot,\top)$ to $(A; \land ,\lor,\bot,\top)$
and  sends any morphism, regarded as a map,  to the same map.  
Heyting algebras are rather special amongst  algebras with reducts 
in $\CCD$ in that the implication~$\to$ is uniquely determined by 
the underlying order. 

The variety $\G$ of G\"odel algebras is the subvariety of $\cat H$
consisting of those algebras which satisfy
the pre-linearity equation $(a\to b)\lor (b\to a)\approx \top$. 
It is a consequence of pre-linearity 
that every subdirectly irreducible G\"odel algebra is a chain. 
Moreover $\G$ is generated as a variety by any infinite Heyting 
chain, and its proper subvarieties are precisely the varieties 
generated by finite chains~\cite{HK}.

Consider the $n$-element chain with elements $0,1,\dots, n-1$ 
labelled so that $0 < 1 < \dots < n-1$.  
Then  we define 
$\Cn =(\{0,1, \dots,n-1\}; \land,\lor,\to, \bot, \top)$,  where 
the constants $\bot$ and $\top$ are taken to be 
the bounds $0$ and $n-1$
and 
\[
a\to b = \begin{cases}
                 \top &\text{if } a \leq b,\\
                  b  & \text{if } b < a.
               \end{cases}
\]
Trivially, every homomorphic image of $\Cn$ is a chain and 
therefore isomorphic to a subalgebra of $\Cn$, whence it follows 
that $\ISP(\Cn)=\HSP(\Cn)$. 
Thus~$\Gn$, defined earlier to be the quasivariety 
generated by~$\Cn$,  is also 
 the variety  generated by 
$\Cn$. 
The lattice of subvarieties of the variety $\G$
 is the chain
\[
\G_1\subseteq \G_2\subseteq\cdots\subseteq\G.
\]
Here $\G_1$ is generated by the trivial algebra and 
$\G_2$ 
 is term-equivalent to the variety of Boolean algebras.

In~\cite{H1969}  it is proved that a Heyting algebra $\A$ is a 
G\"odel algebra if and only if the set  of its prime lattice filters  
forms a forest, that is, the set of prime filters that contain a given 
prime filter forms a chain (see
 \cite{HK}). 
Henceforth, following the natural
 duality approach, we work with homomorphisms into $\two$ 
rather than with
prime filters:   
$\A$ is a G\"odel algebra if and only if  $\CCD(\fnt{U}(\A),\two)$,  ordered 
pointwise,  is a forest.
As we  have noted already, 
$\A$ belongs to 
$\Gn$ if and only if the forest  $\CCD(\fnt{U}(\A),\two)$ has depth 
at most $n-2$. 
  This result is well known but hard to attribute; 
it can be seen as a consequence of 
Lemmas~\ref{lem:isoupchain} and~\ref{lem:chain-hom} below.  
For completeness we 
belatedly 
recall  the definition of depth.
Assume we have a poset $P$  
with the property that for every $p \in P$ the up-set~${\uparrow}p$  
does not contain an infinite ascending chain.  Then 
for $p \in P$ we define 
$$
  d(p)=\Max \{\,|C|-1\mid C\subseteq {\uparrow} p\mbox{ and } C \mbox{ is a chain}\,\}.
$$ 
If  $\{ d(p) \mid p \in P\}$ is bounded above, then the  
\defn{depth} of~$P$ is defined to be $\sup\{ d(p) \mid p \in P\}$.
We note for future use the fact that
in a poset  of finite depth the order relation determines and is determined by the associated covering relation, which we denote by $\newcover$.

An algebra $\A$ in a quasivariety
$\CA = \ISP(\M)$ is  
 determined 
by the homomorphisms from~$\A$ into 
$\M$.   This fact underlies the centrality in natural duality 
theory of the hom-sets $\CA(\A,\M)$ for~$\A \in\CA$.
Accordingly 
we shall  assemble a number of 
results about homomorphisms from an algebra $\A\in \Gn$ into $\Cn$. 
We indicated already in Section~\ref{sec:intro}
the importance of  the $\CCD$-homomorphism 
$\omega\colon
\fnt{U}( \Cn)\to \two$ defined by
$\omega(\top)=1$ and $\omega(k)=0$ for 
$k < \top$. 
 This reflects the key role 
played by the pre-image of the constant $\top$  in 
the 
study of homomorphisms between Heyting algebras. 
 For any $\A \in \Gn$ and any $f \in \Gn(\A,\Cn)$ the 
map  $\w \circ f \in \CCD(\fnt{U}(\A),\two)$ and 
$f^{-1}(\top)=(\w \circ f)^{-1}(1)$.

\begin{lemma}  \label{lem:isoupchain} 
Let $\A\in\Gn$ and $x\in\Gn(\A,\Cn)$.
For each $i\in \ran x\setminus\{0\}$ let $u_i\colon\A\to \two$ 
be defined by  $u_i(a)=1$ if and only if $x(a)\geq i$. 
Then the assignment $i\mapsto u_i$ determines a bijection 
between $\ran x\setminus\{0\}$ and
the subset\, ${\uparrow} (\w\circ x)$ of  $\CCD(\fnt U(\A),\two)$.
In particular   $|{\uparrow} (\w\circ x)|\leq n-1$
for each~$x\in \Gn(\A,\Cn)$.

Moreover, if $x,y \in \Gn(\A, \Cn)$, 
then 
\[\omega\circ x \, \newcover \, \omega \circ y \ \Longleftrightarrow\  
x^{-1}(\top)\subseteq y^{-1}(\top) \mbox{
and }|\ran x| -1=|\ran y|. \]
\end{lemma}
\begin{proof} 
Certainly, for each 
$i\in\ran x\setminus\{0\}$, 
the map $u_i$ 
is a   $\CCD$-homomorph\-ism 
for which  $u_i \geq \w \circ x$. 
To see that the map $i \mapsto u_i$
 is injective,  we argue as follows. 
Let  $i,j \in \ran x\setminus\{0\}$  and assume that $i < j$.  
Then there exists $a\in \A$ for which  $x(a) = j$.  
This implies that $u_j(a) = 1$  and $u_i(a) =0$,  so $u_i \ne u_j$.

Now let $u\in {\uparrow} (\w\circ x)$.  
Let $k=\min\{x(a)\mid a \in \A \text{ and } u(a)=1\}$.  
Note that $k \ne 0$.
We claim that $u = u_k$.
Certainly $u\leq u_k$.
 Now  assume that  $b\in \A$ is such that
$u_{k}(b)=1$, that is,
 $x(b)\geq k$. 
Choose  $a\in\A$  such that $u(a)=1$ and $x(a)=k$.
Then 
 $u(a\to b)\geq(w\circ x)(a\to b)=\w(x(a)\to x(b))=\w(\top)=1$.
It follows that $u(b)\geq u(a\wedge b)=
u(a\wedge(a\to b))=u(a)\wedge u(a\to b)=1$. 
Therefore $u_k\leq u$. 
\end{proof}

 The final claim in  the following lemma  appears in
\cite[Section~3.5]{DW} (see also the proof of \cite[Theorem 2.4]{D76}), but the lemma gives additional information. 
 We shall denote the power set of $\{0, \ldots, n-1\}$ by 
$\powersetn$.

\begin{lemma}  \label{lem:chain-hom} Let $\A\in\Gn$. 
Let 
\[
T =\{\,(u,V) \in \CCD(\fnt U(\A),\two)\times
\powersetn \mid\, 0,n-1\in V\mbox{ and  } \, |{\uparrow} u|+1 =|V| \, \}.
\]  
Then there exist  well-defined  and mutually inverse maps 
\[\iota_{\A} \colon \Gn(\A,\Cn)  \to T\mbox{\quad  and\quad  }
\gamma_{\A} \colon T\to \Gn(\A,\Cn).\]
The  first of these is 
defined by 
$\iota_{\A} \colon x \mapsto  (\w \circ x, \ran x)$, for 
$x \in \Gn(\A,\Cn)$.
The  map $\gamma_{\A}  $ is defined in the following way: 
let 
$V \in \powersetn$ 
be such that
$V =\{ i_0, i_1, \ldots i_m\}$, 
where  
 $i_0 = 0$, $i_m= n-1$   
and 
$ i_j < i_{j+1}$ in $\Cn $ for 
$0\leq j \leq m$.
Then,   for $a \in \A$,  
\[\gamma_{\A}  (u,V)(a)  = i_k, \text{ where  } 
k=|\{ \,v\in {\uparrow}u\mid v(a)=1\,\} |.
\]
In particular, the map  
$x \mapsto \w \circ x$ 
is a surjection from 
$\Gn(\A,\Cn) $ to $\CCD(\fnt{U}(\A),\two)$.
\end{lemma}

 \begin{proof}
Lemma~\ref{lem:isoupchain} tells us that  $\iota_{\A} $ is a map from 
$\Gn(\A,\Cn) $ into~$ T$.
Since $|{\uparrow} u|+1 =|V|$, the map $\gamma_{\A} (u,V)$ is well defined for each $(u,V) \in T$.
It is straightforward to check that 
 $\gamma_{\A} (u,V)$
is  a
homomorphism from~$\A$ to $\C_n$ for each 
$(u,V) \in T$ 
and 
that  $\gamma_{\A} $ and $\iota_{\A} $ are
mutually 
 inverse. 
\end{proof}

Combining the fact that the map 
$x\mapsto \w\circ x$ 
is a 
surjection from $\Gn(\A,\Cn)$ to $\CCD(\fnt{U}(\A),\two)$ with 
Lemma~\ref{lem:isoupchain}, we obtain an alternative proof  of 
 the well-known fact that $|{\uparrow}u|\leq n-1$ 
for each $\A\in\Gn$ and
$u\in \CCD(\fnt{U}(\A),\two)$.
As a consequence 
 of Lemma~\ref{lem:chain-hom}, we can also
describe the endomorphisms of $\Cn$; 
 cf.~\cite[Lemma~2.2]{DT05}.   In particular,  $e \mapsto \ran e$ 
sets up  a bijection from $\End \Cn$ to 
$\{\,V\in\powersetn\mid 0,n-1\in V\,\}$.

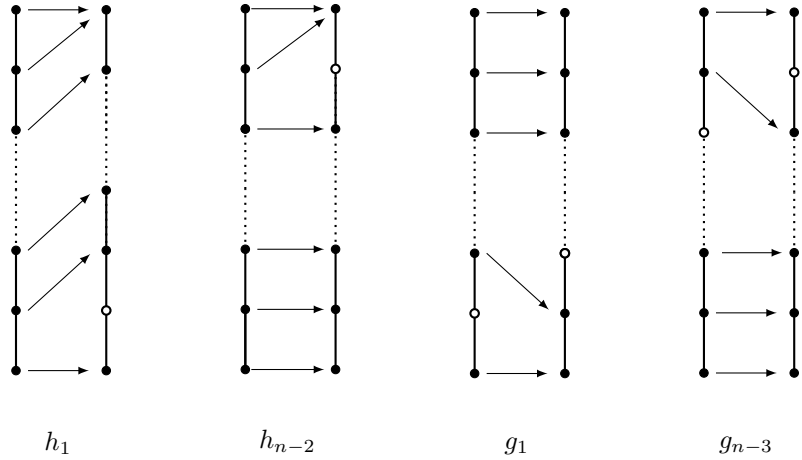
\begin{figure} [ht]
\begin{center}
\begin{tikzpicture}[scale=.8]
\draw[-latex,thin,shorten >= 2pt] (5.2,-1)--(6.3,-1);
\draw[-latex,thin,shorten >= 2pt] (5.2,5)--(6.3,5);
\draw[-latex,thin,shorten >= 2pt] (5.2,0)--(6.3,1);
\draw[-latex,thin,shorten >= 2pt] (5.2,1)--(6.3,2);
\draw[-latex,thin,shorten >= 2pt] (5.2,3)--(6.3,4);
\draw[-latex,thin,shorten >= 2pt] (5.2,4)--(6.3,4.9);

\li{(5,.0)--(5,-1)}
\li{(6.5,0)--(6.5,1)}
\li{(6.5,.0)--(6.5,-1)}
\li{(5,.0)--(5,1)}
\li{(5,4)--(5,5)}
\li{(6.5,4)--(6.5,5)}
\li{(5,3)--(5,4)}
\dotli{(6.5,3)--(6.5,4)}
\dotli{(5,1)--(5,2)}
\dotli{((6.5,1)--(6.5,2)}
\dotli{(5,2)--(5,3)}
\li{(6.5,1)--(6.5,2)}
\dotli{(6.5,2)--(6.5,4)}
\po{5,-1}
\po{6.5,-1}
\po{5,5}
\po{6.5,5}
\po{5,0}
\po{5,1}
\po{5,3}
\po{5,4}
\epo{6.5,0}
\po{6.5,1}
\po{6.5,2}
\po{6.5,4}

\node at (4.2,0) {};
\node at (4.2,1) {};
\node at (4.4,3) {};
\node at (4.2,4) {};
\node at (6.6,0) {};
\node at (6.6,1) {};
\node at (6.8,3) {};
\node at (6.6,4) {};
\node at (5.7,-2.2){$
h_1$};
 \end{tikzpicture} 
\hspace*{.5cm}
\begin{tikzpicture}[scale=.8]
\li{(5,.0)--(5,-1)}
\li{(6.5,0)--(6.5,1)}
\li{(6.5,.0)--(6.5,-1)}
\li{(5,.0)--(5,1)}
\li{(5,4)--(5,5)}
\li{(6.5,3)--(6.5,5)}
\li{(5,3)--(5,4)}
\dotli{(5,1)--(5,2)}
\dotli{(5,2)--(5,3)}
\dotli{(6.5,1)--(6.5,4)}
\po{5,-1}
\po{6.5,-1}
\po{5,5}

\po{5,0}
\po{5,1}
\po{5,3}
\po{5,3}
\po{5,4}
\po{6.5,0}
\po{6.5,1}
\po{6.5,3}
\epo{6.5,4}
\po{6.5,5}

\node at (4.2,0) {};
\node at (4.2,1) {};
\node at (4.4,3) {};
\node at (4.2,4) {};
\node at (6.6,0) {};
\node at (6.6,1) {};
\node at (6.8,3) {};
\node at (6.6,4) {};
\li{(5,.0)--(5,-1)}
\draw[-latex,thin,shorten >=2pt] (5.1,-1)--(6.4,-1);
\draw[-latex,thin,shorten >= 2pt] (5.1,5)--(6.4,5);

\draw[-latex,thin,shorten >= 2pt] (5.2,0)--(6.4,0);
\draw[-latex,thin,shorten >= 2pt] (5.2,1)--(6.4,1);
\draw[-latex,thin,shorten >= 2pt] (5.2,3)--(6.4,3);
\draw[-latex,thin,shorten >= 2pt] (5.2,4)--(6.4,4.9);
\node at (5.7,-2.2){$h_{n-2}$};
\end{tikzpicture} 
\hspace*{.5cm}
\begin{tikzpicture}[scale=.8]
\draw[-latex,thin,shorten >= 2pt] (5.2,-1)--(6.3,-1);
\draw[-latex,thin,shorten >= 2pt] (5.2,5)--(6.3,5);
\draw[-latex,thin,shorten >= 2pt] (5.2,4)--(6.3,4);
\draw[-latex,thin,shorten >= 2pt] (5.2,3)--(6.3,3);
\draw[-latex,thin,shorten >= 2pt] (5.2,1)--(6.3,0);

\li{(5,.0)--(5,-1)}
\li{(6.5,0)--(6.5,1)}
\li{(6.5,.0)--(6.5,-1)}
\li{(5,.0)--(5,1)}
\li{(5,4)--(5,5)}
\li{(6.5,4)--(6.5,5)}
\li{(5,3)--(5,4)}
\li{(6.5,3)--(6.5,4)}
\dotli{(5,1)--(5,2)}
\dotli{((6.5,1)--(6.5,2)}
\dotli{(5,2)--(5,3)}
\dotli{(6.5,2)--(6.5,3)}
\po{5,-1}
\epo{6.5,1}
\po{5,5}
\po{6.5,5}
\po{5,1}
\epo{5,0}
\po{5,3}
\po{5,4}
\po{6.5,0}
\po{6.5,-1}
\epo{6.5,1}
\po{6.5,3}
\po{6.5,4}

\node at (4.2,0) {};
\node at (4.2,1) {};
\node at (4.4,3) {};
\node at (4.2,4) {};
\node at (6.6,0) {};
\node at (6.6,1) {};
\node at (6.8,3) {};
\node at (6.6,4) {};
\node at (5.7,-2.2){$g_1$};
\end{tikzpicture} 
\hspace*{.5cm}
\begin{tikzpicture}[scale=.8]
\draw[-latex,thin,shorten >= 2pt] (5.2,-1)--(6.3,-1);
\draw[-latex,thin,shorten >= 2pt] (5.2,5)--(6.3,5);
\draw[-latex,thin,shorten >= 2pt] (5.2,0)--(6.3,0);
\draw[-latex,thin] 
 (5.3,1)--(6.3,1);

\draw[-latex,thin,shorten >= 2pt] (5.2,4)--(6.3,3);

\li{(5,.0)--(5,-1)}
\li{(6.5,0)--(6.5,1)}
\li{(6.5,.0)--(6.5,-1)}
\li{(5,.0)--(5,1)}
\li{(5,4)--(5,5)}
\li{(6.5,4)--(6.5,5)}
\li{(5,3)--(5,4)}
\li{(6.5,3)--(6.5,4)}
\dotli{(5,1)--(5,2)}
\dotli{((6.5,1)--(6.5,2)}
\dotli{(5,2)--(5,3)}
\dotli{(6.5,2)--(6.5,3)}
\po{5,-1}
\po{6.5,-1}
\po{5,5}
\po{6.5,5}
\po{5,0}
\po{5,1}
\epo{5,3}
\po{5,4}
\po{6.5,0}
\po{6.5,1}
\po{6.5,3}
\epo{6.5,4}

\node at (4.2,0) {};
\node at (4.2,1) {};
\node at (4.4,3) {};
\node at (4.2,4) {};
\node at (6.6,0) {};
\node at (6.6,1) {};
\node at (6.8,3) {};
\node at (6.6,4) {};
\node at (5.7,-2.2){$g_{n-3}$};
\end{tikzpicture} 
\end{center}
\caption{Endomorphisms and partial endomorphisms of~$\Cn$}
\label{fig:pe}
\end{figure}

We shall make use in the next sections of certain 
endomorphisms and partial endomorphisms of $\Cn$. 
For $1\leq i\leq n-2$
we let $h_i$ be the unique endomorphism of $\Cn$ 
with 
$\ran h_i=C_n\setminus \{i\}.$
More precisely,
$h_i(x)=x+1$ if $i\leq x<n-1$ and  $h_i(x)=x$ otherwise.
(These endomorphisms also appear in \cite[Section~2]{CD98}, with $h_i$ denoted~$e_i$,
but the use we make of them is different.)
For $1\leq i \leq n-3$ we define the partial endomorphism 
$g_i$  with domain $C_n \setminus \{ i\}$ as follows:
\[
g_i (x) = \begin{cases}
                x &  \text{if }  x \ne i+1, \\
                i  & \text{if } x = i+1.
\end{cases}
\]
These maps are indeed partial endomorphisms, 
 none of which extends to an endomorphism. 
For $1\leq i \leq n-3$, let  
$f_i\colon C_n \setminus \{ i+1\} \to C_n\setminus \{i\}$ be the 
inverse  of $g_i$.  
For $1\leq i \leq  n-4$ the map $f_i$ is a non-extendable partial 
endomorphism;  $f_{n-3}$
extends to $h_{n-3}$.
Figure~\ref{fig:pe} depicts $h_1$, $h_{n-2}$,  
$g_1$ and $g_{n-3}$; 
corresponding diagrams of  $f_1$ and $f_{n-3}$
are obtained from those of $g_1$ and $g_{n-3}$ by left-to-right 
reflection.   
We fix for future use the following notation:  for $n \geq 4$,
\[
\Sigma_n = \{f_1,g_1\}\times\cdots\times \{f_{n-3},g_{n-3}\}\times
\{ h_1, \ldots, h_{n-2}\}.
\]

In Lemma~\ref{lem:isoupchain} we described, for any given 
$\A \in \Gn$,  the covering relation on the distinct elements of the 
set $\{\, \omega \circ x \mid x \in \Gn(\A,\Cn)\,\}$.  
Below we  complement this result by demonstrating when elements 
$\w \circ x $ and $\w \circ y$ coincide.

\begin{lemma} \label{lem:endo-range} 
Fix $n\geq 4$ and $\sigma\in\Sigma_n$.
Let $\A\in\Gn$ and  let $x,y \in \Gn(\A, \Cn)$ 
be such that  $x \ne y$.    
Then the following statements are equivalent:
\begin{newlist}
\item[{\upshape(1)}] $x^{-1}(\top) = y^{-1}(\top)$;
\item[{\upshape(2)}] $\omega\circ x =\omega \circ y$;
\item[{\upshape(3)}]   there exists  a finite sequence
 $ z_0=x, z_1 ,
\dots, z_N = y$ of elements of $\Gn(\A, \Cn)$ 
with the property that,   for each $0 \leq j< N$,  there is some $i_j\in  \{1,\ldots,n-3\}$ such that
$z_{j+1} = \sigma_{i_j} \circ z_j$ or $z_j = \sigma_{i_j} \circ z_{j+1}$.
\end{newlist} 
\end{lemma}

\begin{proof} 
Conditions  (1) and (2) are equivalent since  
$\omega(k)=\top$ if and only if~$k=~\top$.

Since $f_i$ and $g_i$ are inverses of each other, without loss of 
generality we may assume that $\sigma_i=g_i$ for each 
$i\leq n-3$.
Since $g_i(k) = \top$ if and only if $k = \top$,  (3) implies~(1).
It remains to show that (2) implies (3).  
By Lemma~\ref{lem:isoupchain} and condition (2), 
$|\ran x|=|\ran y|$.
Let 
$m=|\ran x|=|\ran y|$ 
and 
$V=\{0,\ldots,m-2\}\cup\{n-1\}$. 
Now let 
$u=\gamma_{\A} (\w\circ x,V)$
as defined in Lemma~\ref{lem:chain-hom}. 
It is easy to see that either 
$x= u$ 
or there exists a sequence of
 $j_1,\ldots,j_{\ell}$ of elements of $\{1,\ldots,n-3 \}$ such that 
$u=g_{j_1}\circ g_{j_2}\circ\cdots \circ g_{j_{\ell}}\circ  x$. 
Similarly, $y=u$ 
or  
$u=g_{k_1}\circ g_{k_2}\circ\cdots \circ g_{k_{m}}\circ  y$ 
for some $k_1,\ldots,k_{m}\in\{1,\ldots,n-3 \}$. 
Since $x \ne y$ we cannot have both $x =u$ and $y = u$.  
Considering the three remaining possibilities in turn it is easy to see 
that (3) holds in each case.     
 \end{proof}

\section{Natural dualities for $\Gn$}\label{sec:natdual}

In what follows we shall use Priestley duality as an ancillary tool.  
We shall assume familiarity with basic facts concerning this 
prototypical natural duality (to be found in \cite{CD98}
and \cite[Chapter~11]{ILO2}), but we do need to establish notation.  
We denote the  category  of Priestley spaces by $\CP$.  
 We can express $\CCD$ and $\CP$ as, respectively,
$\ISP(\two)$ and $\IScP(\twoT)$, where 
$\two = (\{0,1\}; \land,\lor ,0,1)$ and 
$\twoT = (\{0,1\}; \leq ,\Tp)$;  
here $\Tp$ is the discrete topology and $\IScP(\twoT)$ is the class 
of isomorphic copies of closed substructures of powers of $\twoT$.  
  Here we
shall use non-generic symbols $\fnt{H}$ and $\fnt{K}$ for 
the hom-functors $\fnt{H} = \CCD(-,\two)$  
and $\fnt{K} = \CP(-,\twoT)$ which set up a dual equivalence 
between $\CCD$ and $\CP$. 
 Given $\Lalg\in\CCD$, the evaluation map
 $k_{\Lalg}\colon \Lalg \to  \fnt{KH}(\Lalg)$ is   defined by
$k_{\Lalg}(a) (u) = u(a)$, 
for $a \in \Lalg$ and 
$u \in \fnt{H}(\Lalg)$; 
this map is an isomorphism.  
We refer to the Priestley space $\fnt{H}(\Lalg)$ as the 
\defn{Priestley dual} of $\Lalg$.

We now turn to  natural dualities more generally.
We shall confine attention to the varieties 
$\Gn=\ISP(\Cn) $ that interest us,
referring  to \cite{CD98} any  reader who requires an account in 
a more general setting. 
We note at the outset that 
it will suffice for our purposes  to consider a  more restricted form of alter 
ego than is allowed for in \cite{CD98}.
(We also remark that we have no need in this paper to consider natural dualities
which are multisorted.)
We consider  
a topological structure
$\CnT = (C_n; G, H, \Tp)$, where $G \subseteq \End \Cn$, $H \subseteq \PE \Cn
\setminus \End \Cn$ 
(the (non-total) partial endomorphisms of~$\Cn$), and  $\Tp$ is  the discrete topology.  
We refer to
 $\twiddle{\C_n}$ as an \defn{alter ego} for $\Cn$.
We define 
$\CXn$ 
to be the topological quasivariety  generated by $\CnT$,
 {\it viz.}~$\CXn = \IScP(\CnT)$:  a topological structure of the 
same type as~$\CnT$  belongs   to $\CXn$ if and only if it is 
isomorphic to a closed substructure of a power of $\CnT$;
here operations and partial operations 
are lifted  pointwise.
 The superscript $^+$ serves to indicate that the 
empty structure is included in $\CXn$.  
The morphisms of $\CXn$  are the continuous 
structure-preserving maps. 

We define  hom-functors $\D \colon \Gn \to  \CXn$ and 
$\E \colon \CXn \to \Gn$  as follows:
\begin{alignat*}{2} 
 & \D  \colon \Gn\to \CXn,  \qquad && \hbox{$
          \begin{cases}
                \D (\A)= \Gn(\A,\Cn) \\
                \D (f) = - \circ f,
          \end{cases}$}  \\  
 & \E \colon \CXn \to \Gn,  \qquad && \hbox{$
           \begin{cases}
                 \E (\X)= \CXn(\X,\CnT)\\
                 \E (\phi)  = - \circ \phi ;
           \end{cases}$}  
\end{alignat*}
here
$\Gn(\A,\Cn)$ is considered as a substructure of $(\twiddle{\C_n})^A$
and $\CXn(\X,\CnT)$ inherits its algebra structure pointwise from~$\Cn$.
A crucially important fact is that these 
 functors are well defined. 
This is a consequence of our assumption that we include in $\CnT$
only operations and partial operations which are 
algebraic.
 Moreover, for each  $\A \in \Gn$,  
the evaluation map  $e_{\A}$, given by 
$e_{\A} (a)(x) = x(a)$ (for $a \in \A$ and $x \in \D(\A)$), is an 
embedding from $\A $ to $\E\D(\A)$.   
Likewise, for each $\X \in \CXn$, 
the map $\varepsilon_{\X}$ given by 
$\varepsilon_{\X}(\phi)(\alpha) = \alpha(\phi)$ (for  $\phi \in \X$ 
and $\alpha\in\E(\X)$) is an embedding. 
In categorical terms, $(\D,\E, e,\varepsilon)$ is a dual adjunction 
between $\Gn$ and $\CXn$ with the unit and counit of the 
adjunction given by the evaluation maps.
 (See \cite[Chapter~2]{CD98} for a justification  of 
these assertions in a general setting.)    
Let $\A \in \Gn$.  We say that $\CnT$  (or just $G \cup H$) \defn{yields a duality on} $\A$ 
if $e_{\A}$ is an isomorphism from $\A$ to $\E\D(\A)$ and that 
$G\cup H$ \defn{yields a duality on} $\Gn$ 
if it yields a duality on each $\A \in \Gn$.    
For later  use, we say   that a dualising alter ego 
$\CnT$ \defn{yields a full duality on} $\Gn$ if 
$\D\E(\X) \cong \X$ for all $\X \in \CXn$. 

We shall need the following result. It is obtained by specialising the 
Test Algebra Lemma to the very  particular situation that concerns 
us.
See  \cite[Section~8.1]{CD98} for the   general version of this result  
and contextual discussion.
We reiterate that $\Gn$ is endodualisable so that the assumptions 
of Lemma~\ref{lem:TAL}  are met when $G=\End \Cn$.

\begin{lemma}
\label{lem:TAL}{\bf (Test Algebra Lemma, special case)}
Let $(C_n;G,  H,\Tp)$  be an alter ego of $\Cn$ which is such that 
$G \subseteq \End \Cn$, $H \subseteq \PE \Cn\setminus \End \Cn$
and $G \cup H$ yields a duality on 
$\Gn$.
Let $e \in G$.  
Then $ (G \setminus \{e\}) \cup H $ yields a duality on $\Gn$ provided it 
yields a duality on the single algebra~$\Cn$. 
\end{lemma}

\begin{proof}    
The Test Algebra Lemma  in its general form 
 tells us that we can discard~$e$ from $G$ without destroying the 
duality so long as $ (G \setminus \{e\})\cup H$ 
yields a duality on  
$\graph e$,  regarded as an algebra in $\Gn$.  
But the graph of any endomorphism is isomorphic to $\Cn$.  
\end{proof}

We contrast the use of $\Cn$ as a test algebra with that employed   
in \cite[proof of Theorem~2.4]{DT05}.   
There Davey and Talukder  identify a particular 
generating set $G$ 
 for $\End \Cn$.   
They then show that the  duality it yields is optimal 
by showing that $G
 \setminus \{ e\}$ does not yield a 
duality on the particular algebra $\Cn$, for any $e \in G$.  
This means that they are using the Test Algebra Lemma to  guide 
the choice of an  algebra that witnesses  indispensability  of each 
member of their set~$G$.  
We use the Test Algebra Lemma in the opposite  direction: 
the lemma   tells us that to prove that a given endomorphism  $e$ 
can be dropped from a dualising alter ego  it suffices to test
this on the \emph{single} algebra  $\graph e$---we do not have to 
verify that $(G \setminus \{ e\}) \cup H$ yields a duality on 
\emph{every} $\A \in \Gn$.      

Henceforth, unless indicated otherwise,  
we shall consider  varieties $\Gn$ 
for which $n\geq 4$.
 We include the endomorphism~$h_1$  
(as defined in Section~\ref{sec:prelim})
in our alter ego for $\Cn$,
rather than any alternative  endomorphism, because
this  makes it particularly easy to establish Claim~4
of the proof of Proposition~\ref{prop:piggy}.  We  adopt a more even-handed attitude to 
endomorphisms in Theorem~\ref{thm:newduality}. 
The proof of  the proposition 
draws very heavily on  the ideas
 used to prove the Piggyback Duality Theorem 
\cite[Theorem~7.2.1]{CD98},  as this  applies to a quasivariety  
$\ISP(\M)$, where $\M$ is a finite algebra with a reduct in $\CCD$.

\begin{proposition}  \label{prop:piggy}
Let the partial endomorphisms $g_1, \ldots, g_{n-3}$  and 
endomorphism $h_1$ be defined as in Section~{\upshape\ref{sec:prelim}}.
Then  $\{ g_1, \ldots g_{n-3},h_1\}$  
yields a duality on the algebra $\Cn$.
\end{proposition}

\begin{proof}  
Observe that the evaluation map  $e_{\Cn} \colon \Cn \to \E\D(\Cn)$ is 
injective, and  the evaluation map  
$k_{\fnt{U}(\Cn)} \colon  \fnt{U}(\Cn) \to \fnt{KHU}(\Cn)$ is an 
isomorphism, and so surjective.  
We want to show that $e_{\Cn}$ is surjective.  
Now we bring in the critical, but entirely elementary, observation 
that it will suffice to construct an injective map 
$\Delta \colon \fnt{ED}(\Cn) \to \fnt{KHU}(\Cn)$
(see \cite[proof of Piggyback Duality Theorem~7.2.1]{CD98} or 
\cite{DW}). 

Recall that  $ \omega \colon\fnt{U}(\Cn) \to \two$  denotes  the 
$\CCD$-morphism with ${\omega^{-1}(1) = \{n-1\}}$
and that
for each
$u \in \fnt{HU}(\Cn)$ 
we can find 
$x \in \D(\Cn)$ 
such that  
$u = \omega \circ x$.
 We may now attempt to define $\Delta $ as follows. 
 Given $\phi \in \E\D(\Cn)$ let 
\[
\bigl(\Delta(\phi)\bigr)(u)=\bigl(\Delta(\phi)\bigr)(\omega \circ x) =\omega (\phi(x)) .
\]
We now establish a series of claims.  
These combine with the observations above to prove the 
proposition.

\begin{claimlist}
\item[{\bf 1.\!\!}]  $\Delta$ is a  well-defined map.

We have already observed  that every element of $\fnt{HU}(\Cn)$ is 
of the form $\omega \circ x$ for some $x \in \D(\Cn)$. 
We must now check that, for $x$ and $y$ in $\D(\Cn) = \End\Cn$ 
and $\phi \in \fnt{ED}(\Cn)$,
\[
\omega \circ x = \omega \circ y 
\Longrightarrow \omega(\phi(x)) = \omega(\phi(y)).
\]
Suppose first that  
$y = g_i \circ x$ 
for some $g_i$.  Then 
$\phi(x) = \top $ if and only if 
$\phi(y)=
g_i(\phi(x)) = \top.$ 
We argue likewise when 
$x = g_i \circ y$.
Hence, by Lemma~\ref{lem:endo-range},
 $\phi(x) = \top $ if and only if
$\phi(y)= \top$.
Since  
$\omega(j) = 1$ if and only if $j =\top$,
our claim is proved.

\item[{\bf 2.\!\!}]  $\Delta(\phi)$ is order-preserving for each 
$\phi \in \E\D(\Cn)$.

For  $1\leq i\leq n-1$, 
let $u_i\colon \Cn\to\two$ be the map determined by 
$u_i^{-1}(1)={\uparrow}i$.   
It is trivial to check that  the set $\{u_1,\ldots,u_{n-1}\}$  coincides 
with $\fnt{HU}(\Cn)$ and $u_{1}>u_2> \cdots> u_{n-1}$.  

Let $1 \leq i  < j < n-1$, so  $u_j < u_i$. 
 Assume that $(\Delta(\phi))(u_j)= 1$.  
We wish to show that $(\Delta(\phi))(u_i)= 1$.
For each $k$ such that ${i\leq k\leq j}$, let
${x_k=\gamma_{\A} (u_k,(\,\{0\}\cup\{n-k\ldots,n-1\}\,))}$, 
where the map $\gamma_{\A} $ is as defined 
in Lemma~\ref{lem:chain-hom}. 
It follows that $\w\circ x_i=u_i$ and $\w\circ x_j=u_j$. 
Then ${1=(\Delta(\phi))(u_j)=\w(\phi(x_j))}$, that is, 
$\phi(x_j)=\top$. 
Clearly $x_{k}=h_1\circ x_{x+1}$  whenever
$i\leq k<j$. 
Since $\phi(x_j)=\top$,  
$h_1(\top)=\top$, and $\phi$ preserves $h_1$, 
it follows that $\phi(x_i)=\top$. 
Therefore $(\Delta(\phi))(u_i)=\w(\phi( x_i))=1$.

\item[{\bf 3.\!\!}] For each $\phi \in \E\D(\Cn)$ the map  
$\Delta(\phi)\colon \fnt{HU}(\Cn) \to \twoT$ is continuous.

This is immediate because $\Cn$ is finite.  

\item[{\bf 4.\!\!}]  $\Delta$ is injective.   

Suppose that $\phi, \psi \in \E\D(\Cn)$ with $\phi \ne \psi$.
Pick $x \in \D(\Cn)$ such that $\phi(x) \ne \psi(x) $ 
in $\C_n$.  
Without loss of  generality,  assume that $\phi(x) < \psi(x)$. 
Let $j=(n-1)-\psi(x)$.  
Then $h_1^{j}(\psi( x))=\top$ and $h_1^{j}(\phi( x))\neq \top$ 
(where $h_1^{j}$ denotes the $j$-fold composition of $h_1$ if 
$j>0$ and  the identity map if $j=0$). 
Since $\phi$ and $\psi$ preserve $h_1$, 
\[
(\Delta(\psi))(\w\circ h_1^{j}\circ x)=\w(\psi(h_1^{j}\circ x))=
\w(h_1^{j}(\psi( x)))=\w(h_1^{j}(\psi( x)))=1
\]
and
\[
(\Delta(\phi))(\w\circ h_1^{j}\circ x)=\w(\phi(h_1^{j}\circ x))=
\w(h_1^{j}(\phi( x)))=\w(h_1^{j}(\phi( x)))=0.
\]
Therefore  $\Delta(\phi)\neq\Delta(\psi)$.\qedhere
\end{claimlist}
\end{proof}

The following theorem 
 supplies  a family of alter egos each of which dualises $\Gn$.
In Section~\ref{Sec:Fullness}, we shall  see that,  even if the natural 
dualities presented in Theorem~\ref{thm:newduality} are closely 
connected, they have significantly different properties. 
We recall that the definition of $\Sigma_n$ was given in Section~\ref{sec:prelim}.

\begin{theorem} \label{thm:newduality}  
Let $\sigma \in \Sigma_n$.  
Then 
${\twiddle{\Cn^\sigma}} = (C_n; \sigma,\Tp)$
yields a duality on ${\Gn =\ISP(\Cn)}$.
\end{theorem}

\begin{proof}
We first note that Lemma~\ref{lem:TAL} and 
Proposition~\ref{prop:piggy}
combine to tell us that the alter ego 
$(C_n;g_1,\ldots,g_{n-3},h_1,\Tp)$ yields a duality on ~$\Gn$.

For any~$i$, the maps $g_i$ and~$f_i$ are interchangeable 
 because their  graphs are mutual converses. 
We note that  
$h_1=f_1\circ \cdots\circ f_{i-1}\circ h_{i}$ for 
$2\leq i\leq n-2$. 
Hence (see \cite[Section~2.4]{CD98})
$h_1$ is entailed by $f_1,\ldots,f_{n-3}$ and $h_i$.  
Therefore
$\twiddle{\Cn^\sigma}$ yields a duality for any choice of 
$\sigma$ from $\Sigma_n$.
\end{proof}

We remark that we could use the Test Algebra Lemma to prove
that each of the dualities presented in Theorem~\ref{thm:newduality} is optimal;
cf.~\cite[Theorem~2.4]{DT05}.  The technique is standard and we do not include details here.

We briefly consider $\G_3$.  Here 
$\End\C_3=\{\id_{\C_3}, h_1\}$,  
and there 
are no non-extendable endomorphisms to consider.  
We could define
$\Sigma_3=\{h_1\}$ and so bring $n=3$ within the scope of 
Theorem~\ref{thm:newduality}.  But this adds nothing that is  new: 
  already in \cite{D76} the  alter ego $(C_3;h_1,\Tp)$ was shown to 
yield a duality on $\G_3$.     For $n=2$ there is even less that is 
worth saying, since 
$\End (\C_2) = \{\id_{C_2}\}$ 
and there are no non-extendable 
partial endomorphisms.  The duality for $\G_2$ associated  with 
$\Sigma_2$, defined to be $\emptyset$,
is  just Stone duality for Boolean algebras.

 \section{From 
natural duality to 
Priestley/Esakia duality and back again}\label{sec:Translation}

The main objective in this section
is to investigate how the dualities 
presented in Theorem~\ref{thm:newduality} facilitate translation 
from the categorically well-behaved natural duality set-up  to the 
more pictorial representation afforded by Priestley/Esakia duality for 
Heyting algebras.  
Before demonstrating how the translation operates we 
briefly recall the Priestley/Esakia duality.
This  has a long history, and has been rediscovered and 
reformulated many times.
By way of reference we  note here Esakia's paper \cite{Esa}
 and also the 
recent paper \cite{BBGK}.

The relative pseudocomplement in a Heyting algebra is uniquely 
determined by the underlying lattice order. 
More precisely,  we may assert that the forgetful functor 
$\fnt{U}\colon \cat H\to \CCD$ is faithful and 
$\fnt{U}\colon \cat H\to \fnt{U}(\cat H)$ is part of a categorical 
equivalence (actually an isomorphism); the inverse  
$\fnt{V}\colon \fnt{U}(\cat H) \to \cat H$ maps each bounded 
distributive lattice $\Lalg$  that admits a relative 
pseudocomplement to the unique Heyting algebra $\A$ such that 
$\fnt{U}(\A)=\Lalg$. 

An algebra $\Lalg \in \CCD$ can be identified with its second dual 
$\fnt{K}(\X)$, where  $\X = \fnt{H}(\Lalg)$.   
There exists a Heyting algebra $\B$ with 
$\fnt{U}(\B) = \fnt{K}(\X)$  if and only if the Priestley space 
$\X = (X; \leq,\Tp)$ has the property that  ${\downarrow}O$ is  
$\Tp$-open whenever $O$ is $\Tp$-open;  
if this condition is satisfied we call $\X$ an \defn{Esakia space}.
Given Esakia spaces~$\X$ and~$\Y$, a  continuous 
order-preserving 
map $\phi \colon \Y \to \X$ is such that 
$\fnt{K} (\phi) \colon \fnt{K}(\X) \to \fnt{K}(\Y)$ preserves the 
relative pseudocomplement if and only $\phi$ is an 
\defn{Esakia morphism}, meaning that   
$\phi({\uparrow}y) = {\uparrow}\phi(y)$ for all $y \in Y$.  
In summary, there is a dual equivalence between the category of 
Heyting algebras and the category of Esakia spaces 
(with Esakia morphisms), obtained by restricting
the duality between $\CCD$ and $\CP$ to the subcategory
 $\fnt{U}(\cat H)$ and a certain subcategory $\cat E$ of $\CP$.

As observed earlier,  a Heyting algebra is a G\"odel algebra if and 
only if the associated Esakia space is a forest. 
In our formulation of the duality  trees grow downwards.
Restricting  the functors~$\fnt{HU}$ and 
$\fnt{VK}{\restriction}_{\cat E}$  we obtain a dual equivalence 
between the category $\cat G$ of G\"odel algebras  and the 
category $\cat F$ of  Esakia spaces whose order structure is a forest 
and  Esakia morphisms.

The category~$\Gn$ is dually equivalent to the full 
subcategory~$\CFn$ of Esakia spaces whose objects are forests of 
depth at most $n-2$. 
Figure~\ref{Fig:GodelDual}  summarises the various dual equivalences relating to Priestley/Esa\-kia duality and their restrictions to full subcategories, shown by unlabelled vertical arrows.

 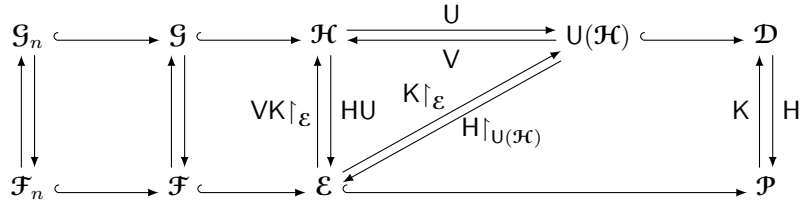
\begin{figure}  [ht]
\begin{center}
\begin{tikzpicture} 
[auto,
 text depth=0.25ex,
 move up/.style=   {transform canvas={yshift=1.9pt}},
 move down/.style= {transform canvas={yshift=-1.9pt}},
 move left/.style= {transform canvas={xshift=-2.5pt}},
 move right/.style={transform canvas={xshift=2.5pt}}] 
\matrix[row sep= 1.5cm, column sep= 1.4cm]
{ 
\node (Gn) {$\Gn$}; & \node (G) {$\cat G$}; &\node (H){$\cat H$};
 && \node (DH) {$\fnt{U}(\cat H)$}; &\node (D){$\CCD$};
\\ 
\node (Fn) {$\CFn$};& \node (F) {$\cat F$}; & \node (E) {$\cat E$};
& &&  \node (P) {$\CP$};
\\ 
};
\draw [-latex , move up] (H) to node  [yshift=-2pt]  {$\fnt U$}(DH);
\draw [latex-, move down] (H) to node [swap] {$\fnt V$}(DH);
\draw [-latex, move right] (H) to node   {$\fnt{HU}$}(E);
\draw [latex-, move left] (H) to node [swap] [xshift=2pt]{$\fnt {VK}\restrict{\cat E}$}(E);
\draw [-latex, move right] (G) to node {}  
(F);
\draw [-latex, move down] (DH) to node  [yshift=+5pt] {$\fnt H
{\restriction}_{\fnt{U}(\cat H)}$}(E);
\draw [latex-, move up] (DH) to node [swap] [xshift=2pt,yshift=-1pt]{$\fnt K{\restriction}_{\cat E}$}(E);
\draw [move down, right  hook-latex] (DH) to node {} (D);
\draw [right hook-latex, move down] (E) to node {} (P);
\draw [latex-, move left] (G) to node 
{} 
(F);
\draw [-latex, move right] (Gn) to node {} 
(Fn);
\draw [latex-, move left] (Gn) to node {} 
(Fn);
\draw [right hook-latex,move down] (G) to node {} (H);
\draw [right hook-latex,move down] (Gn) to node {} (G);
\draw [right hook-latex,move down] (F) to node {} (E);
\draw [right hook-latex,move down] (Fn) to node {} (F);
\draw [-latex, move right] (D) to node   {$\fnt H$}(P);
\draw [latex-, move left] (D) to node [swap] {$\fnt K$}(P);

\end{tikzpicture}
\end{center}
\caption{Priestley/Esakia duality for G\"odel algebras}\label{Fig:GodelDual}
\end{figure}

 For fixed 
$n \geq 4$ 
and each choice of  $\sigma \in \Sigma_n$,   we shall use
$\D^{\sigma }$ and $\E^{\sigma }$ to denote the functors determined by the alter ego $\twiddle{\Cn^{\sigma }}$ of $\Cn$.
Our immediate  aim is to relate the dual space $\D^{\sigma }(\A)$
 to the Priestley/Esakia dual  $\fnt{HU}(\A)$.  
Some word of explanation is needed before we 
demonstrate how to do this.
Let $\cat{Y}^{\sigma }$ denote the full subcategory of $\IScP(\twiddle{\Cn^{\sigma }})$ whose class of objects is  $\mathbb{I}(\fnt{D}^{\sigma }(\Gn))$. 
From Theorem~\ref{thm:newduality} and the 
Priestley/Esakia duality for $\Gn$, it is straightforward to see that 
$\CFn$ and $\cat Y^{\sigma }$
 are equivalent categories.     
Therefore one may ask: why  present a description of 
$\D^{\sigma }(\A)$ from $\fnt{HU}(\A)$,  and vice versa, if this can 
be obtained 
using $\A$ as a stepping stone?
The answer is that to prove the trivial fact 
that $\CFn$ and $\cat Y^{\sigma }$ are equivalent is not our final 
goal. 
 We want to reveal  the very special connection between these 
categories which will be our primary tool in the final sections  of the 
paper.

Assume we have any finitely generated 
(quasi)variety $\CA$  of distributive lattice-based algebras with 
forgetful functor $\fnt{U}\colon \CA \to \CCD$.
In \cite[Section~2]{CPcop} we presented a
 procedure for passing from the natural dual  of an algebra 
$\A \in \CA$ to the Priestley dual $\fnt{HU}(\A)$ when the natural 
duality under consideration was obtained by the piggyback method.  
Here we carry out an analogous process, 
but now based  on any of the 
 dualities we 
established in Theorem~\ref{thm:newduality}.
We shall do this by  proving a variant of \cite[Theorem~2.4]{CPcop}, 
as this theorem applies to the special case in which $\CA=\Gn$.  
This result---Theorem~\ref{thm:RevEngGodel}---achieves more  
than  the direct specialisation of the general result.  
The reason for this lies in the way in which, for $\A \in \Gn$, the
layers  of  the Priestley space $\fnt{HU}(\A)$  
are derived from the action of the maps $g_i$ (or $f_i$) on 
$\D(\A)$, and how the lifting of the chosen endomorphism~$h_j$ 
relates these layers.
(It is convenient to  visualise   
 $\fnt{HU}(\A)$   as being comprised of layers, each layer consisting of the elements at a particular depth; see Fig.~\ref{fig:C5} relating to Example~\ref{Ex:Dual}.)

Before we begin we recap on  the form taken by  
the objects of the dual  category 
$\CXn^{\sigma } = \IScP(\twiddle{\Cn^{\sigma }})$, where 
 $\sigma \in \Sigma_n$.
  These are topological structures
$\X = 
(X;\sigma _1^{\XX},\sigma _2^{\XX},\ldots,\sigma _{n-2}^{\XX},\Tp^{\XX})$,
 where  the partial operations $\sigma _i^{\XX}$ 
(for $1 \leq i\leq n-3$) and the operation  $\sigma _{n-2}^{\XX}$ 
are obtained by pointwise lifting of the corresponding operations 
$\sigma _i$ on~$\Cn$ and the domain of each partial operation is a 
closed substructure of~$\X$ 
(see \cite[Chapter~2]{CD98} for details). 
Let $\newapprox_{\XX}^\sigma $ be  the binary relation defined 
on~$X$ by $x \newapprox_{\XX}^\sigma  y$ if and only if either 
$x=y$ or there exists  a sequence 
$ z_0=x, \dots, z_N = y\in \X$ 
such that,   for each $j\in \{0,\ldots,N-1\}$, there exists
 $i_j\in\{1,\ldots,n-3\}$ such that 
$z_{j+1} = \sigma _{i_j} ^{\XX}\circ z_j$ or 
$z_j = \sigma _{i_j}^{\XX} \circ z_{j+1}$.
Then $\newapprox_{\XX}^\sigma $ is an equivalence relation 
on~$X$.  
The definition of $\newapprox_{\XX}^\sigma$ is motivated by 
Lemma~\ref{lem:endo-range}, which can be recast as follows.

\begin{lemma}\label{Lem:Newapprox}
Let  $\sigma\in\Sigma_n$.  Let $\A\in\Gn$ and  $\X=\D^{\sigma}(\A)$ and $x,y \in \X$.  
    Then 
\[
x^{-1}(\top) = y^{-1}(\top)\ \Longleftrightarrow \ 
\omega\circ x =\omega \circ y \
\Longleftrightarrow\
x\newapprox_{\XX}^\sigma y. 
\]
\end{lemma}

The cluttered notation adopted below is temporarily necessary  because we  shall work simultaneously  with 
more than one alter ego.  
We denote the equivalence class of $x \in X$ 
under ${\newapprox}_{\XX}^\sigma$
by $[x]_{\newapprox_{\XX}^\sigma}$.
We now define a relation 
$\newcover_{\XX}^\sigma $ on $X/{\newapprox}_{\XX}^{\sigma }$ as follows:
\[
[x]_{\newapprox_{\XX}^\sigma}  \, \newcover_{\XX}^ {\sigma }  \,
[y ]_{\newapprox_{\XX}^\sigma}
\Longleftrightarrow  
x  \not{\!\!\newapprox}_{\XX}^\sigma \,  y \text{ and }  
\exists z \bigl(  x \newapprox_{\XX}^\sigma   z\text{ and } 
\sigma_{n-2}^{\XX}(z)
\newapprox_{\XX}^\sigma   y 
\bigr),
\]
and let $\leq_{\XX}^\sigma $ be the reflexive, transitive 
 closure  of 
$\newcover_{\XX}^\sigma $.  
Taking the reflexive, transitive 
closure of the antisymmetric relation 
$\newcover_{\XX}^\sigma$ does not destroy antisymmetry, 
so $\leq_{\XX}^\sigma $ is a partial order.

\begin{lemma}\label{Lem:InvRevEn}
  Let  $\sigma ,\tau \in\Sigma_n$. 
Let $\A$ be an algebra in $\Gn$ and let  $\X=\D^{\sigma }(\A)$ 
and $\X'=\D^{\tau }(\A)$ be the associated dual spaces. 
Then 
\begin{newlist}
\item[{\rm (i)}] $\newapprox_{\XX}^{\sigma }$ and 
$\newapprox_{\XX'}^{\tau }$ are equal;   
 \item[{\rm (ii)}] $\leq_{\XX}^\sigma$ and $ \leq_{\XX'}^{\tau }$ are 
equal.
\end{newlist}

Moreover, for any $\sigma \in \Sigma_n$, the relation 
$\leq_{\XX}^\sigma$ is a partial order on 
$X/{\newapprox}_{\XX}^\sigma$ of depth at most $n-2$,   
 for which ${\newcover}_{\XX}^\sigma$ is the associated covering 
relation. 
\end{lemma}
\begin{proof}

(i) follows directly from Lemma~\ref{Lem:Newapprox}.
Since  we now know  that the equivalence relations on 
$X=\Gn(\A,\Cn)$ obtained from $\sigma$ and $\tau$ are the same 
we shall write simply~$\newapprox$ for the relation and 
$[x]_{\newapprox}$ for the equivalence class of an element~$x$ 
in~$X$.

We now prove (ii).
Let $x, y \in X $  be such that 
$[x]_\newapprox\, {\newcover}_{\XX}^\sigma \, [y]_{\newapprox} $.
Then $x\not\newapprox y$ and there exists 
$z\in[x]_{\newapprox}$ for which 
$\sigma _{n-2}^{\XX}(z)\in[y]_{\newapprox}$. 
Since $h_j(n-2) = n-1$ for any   $j\in\{1,\ldots,n-2\}$, it follows 
that, for $a \in \A$, 
\begin{align*}
(\w\circ \sigma _{n-2}^{\XX}(z))(a)=1& 
\Longleftrightarrow z(a)\in\{n-2,n-1\} 
\Longleftrightarrow(\w\circ \tau _{n-2}^{\XX}(z))(a)=1.
\end{align*}
By Lemma~\ref{Lem:Newapprox},
 ${\tau }_{n-2}^{\XX'}(z)\newapprox \sigma _{n-2}^{\XX}(z)
\in [y]_{\newapprox}$.
So  $x\, {\newcover}_{\XX'}^{\tau }\,y$. 
We deduce that $\newcover_{\XX}^{\sigma }$ and 
$\newcover_{\XX '}^{\tau }$ are equal. 
Therefore $\leq_{\XX}^\sigma$ coincides with~$\leq_{\XX'}^{\tau}$.

The final assertions follow  from Lemma~\ref{lem:isoupchain}
and the way in which the order on $X/{\newapprox}$ is defined.
\end{proof}

\begin{theorem} \label{thm:RevEngGodel} 
Let $\sigma \in\ \Sigma_n$.
Let  $\A\in\Gn$
and $\X=\D^{\sigma }(\A)$ be
its dual space.  
Let $Y= X/{\newapprox_{\XX}^{\sigma }}$ and 
let $\Tp$ be the 
quotient topology  derived from the topology  of\, 
$\D^{\sigma }(\A)$. 
Then $\Y = (Y;\leq_{\XX}^{\sigma },\Tp)$ is a Priestley space 
isomorphic to $\fnt{HU}(\A)$.
\end{theorem}

\begin{proof}  
By Lemma~\ref{Lem:InvRevEn}  we may assume that 
$\sigma =(g_1,\ldots,g_{n-3},h_1)$.   
We shall write 
 $\newapprox$ in place of $\newapprox_{\XX}^{\sigma}$
and omit subscripts from equivalence classes and from the order and covering relations on $X/{\newapprox}$. 

We know that
the map $\Phi_\w \colon x \mapsto \omega \circ x$ from $\D(\A)$ to 
$\fnt{HU}(\A)=(Z;\leq_{\ZZ},\Tp^{\ZZ})$ is surjective. 
Arguing just as in the proof of
\cite[Theorem~2.3]{CPcop} we proved that 
$(Z;\Tp^{\ZZ})$ is homeomorphic to the quotient space 
$(X/\ker(\Phi_\w); \Tp^{\XX}/\ker(\Phi_\w))$.
From the definition of~$\Phi_\w$, we have  $\Phi_\w(x)=\Phi_\w(x)$ 
if and only if $\omega\circ x=\omega \circ y$. 
By Lemma~\ref{Lem:Newapprox},  $\ker(\Phi_\w)$ coincides with  the 
relation $\newapprox$ described  in terms of the liftings of
$g_1,\ldots,g_{n-3}$. 
So we have identified $(\D(\A)/{\newapprox}\, ; \Tp)$ 
with $(Z;\Tp^{\ZZ})$.

It remains to reconcile the order  of the quotient space with that of 
${\fnt{HU}(\A)}$.  
Since we are working with posets of finite depth it suffices to 
consider the covering relations.
First suppose that $[x]\, \newcover\,  [y]$.  
So there exists $z$  such that
$x \newapprox z$, $y \newapprox h_1^{\XX}(z)$.  
Since $i\leq h_1(i)$ for each $i\in\C_n$, we have, for  $a\in\A$, 
\[
\omega(x(a))=\omega( z(a))\leq \omega(h_1( z(a)))=(\omega\circ h_1^{\XX}(z))(a))=
\omega(y(a))
\]
and hence $\omega\circ x\leq \omega\circ y$ in $\fnt{HU}(\A)$.

Conversely, assume that 
$\omega \circ y$ covers $\omega \circ x$ in $\fnt{HU}(\A)$.
Assume   $d(\omega\circ y)=j$ and $d(\omega\circ x)=j+1$.
By Lemma~\ref{lem:chain-hom},  there exists  $z\in\Gn(\A,\Cn)$ 
such that $\omega \circ z=\omega \circ x$ and 
$\ran(z)=\{0\}\cup {\uparrow}(n-1-j)$. 
By Lemma~\ref{Lem:Newapprox}, $z\newapprox x$. 
Now observe that 
$\ran(h_1^{\XX}\circ z)=\{0\} \cup {\uparrow}(n-j)$
and so 
$\omega \circ x=\omega \circ z<\omega \circ (h_1^{\XX}(z))$. 
Since ${\uparrow}(\omega \circ x)$ is a chain and 
$d(\omega \circ (h_1^{\XX}(z)))=j 
=d( \omega \circ
y)$, 
we have $\omega \circ (h_1^{\XX}(z))= \omega \circ x$.
 Therefore $h_1^{\XX}(z)\newapprox y$.  Hence 
$[x]\leq [y]$.
\end{proof}

A retrospective look at  \cite{D76} is  due here. 
There are clear 
similarities between our proof of 
Theorem~\ref{thm:RevEngGodel} and Davey's  original proof of
endodualisability of $\Gn$ 
\cite[Theorem~2.4]{D76}.
 Lemma~\ref{Lem:Newapprox} establishes 
that, for each 
$\A\in\Gn$, our 
relation~$\newapprox_{\D(\A)}^{\sigma}$ coincides 
with the relation $R$ defined in the proof of \cite[Theorem~2.4]{D76}.
But there is an important point to note.
The relations 
$\newapprox$ and $\newcover$ are defined using the lifting of the 
(partial and non-partial) operations $\sigma_i$. Therefore, they are 
available  in  
every space
 in
$\CX=\IScP(\twiddle{\C_n^{\sigma}})$, and not only  
those of the form $\D(\A)$ for some $\A\in\Gn$. 
This difference becomes crucial in the following sections 
when we 
determine which dualities are full and which $V$-formations admit 
amalgamation.

We now take a break from theory to discuss  how  translation works in practice.

\begin{ex}\label{Ex:Dual}
Fix $n=5$ and $\sigma=(g_1,g_2,h_3)$.
We illustrate the passage from the natural dual to the 
Priestley/Esakia  dual for the algebra $\str{C}_5$. 
Here the elements of $\X = \D^{\sigma}(\str{C}_5)$ are exactly the 
endomorphisms of $\str{C}_5$, on which $g_1$, $g_2$ and 
$h_3$ act by composition.   
We label each endomorphism~$e$ of $\str{C}_5$ by writing  
$\ran e \setminus \{0,4\}$ (which uniquely determines~$e$) 
as a string, as indicated in Fig.~\ref{fig:C5}. 
 For endomorphisms $e$ and $f$ we have 
$(e,f) \in \graph g_i^\XX$ if and only if 
\[
 \forall  a\in \str{C}_5\ \bigl(  (e(a),f(a)) \in 
\graph g_i  \bigl)
\]
for $i =1,2$.  In order for this to hold it is necessary that 
$i \notin \dom e$ and $i+1 \notin \ran f$.  
In the figure,  the   solid arrows  
arrows indicate the action of~$h_{3}$.
Dashed   and dotted arrows indicate, respectively, 
the action of  $g_1$ and of~$g_2$.

\begin{figure} [ht]
\hspace*{-1cm}   
\begin{tikzpicture}[scale=.8,
infobox/.style={
shape=ellipse,
minimum size=5mm,
 thin,
draw=black
}
]

\draw [-stealth,decorate,decoration={snake,amplitude=1pt,segment length=5pt,
                pre length=3pt,post length=3pt}]
(2.2,3)-- node [yshift=.5cm]  {$x \mapsto [x]_{\newapprox}$} (5.8,3);



\node (lev0) at (8,6) [infobox] {$\scriptstyle \emptyset$};
\node (lev1) at (8,4) [infobox] {
$\scriptstyle  1 \ 2 \ 3 $}
;
\node (lev2) at (8,2) [infobox] { 
$\scriptstyle 12 \ 
13\ 
23$} 
;
\node (lev3) at (8,0) [infobox] {$\scriptstyle 123$};
\draw [-] (lev0)--(lev1)--(lev2)--(lev3);
;

\node (HU5) at (8,-1.5) {$\fnt{HU}(\str{C}_5)$};
\node (D5) at (-1,-1.5)  {$\fnt{D}^\sigma (\str{C}_5)=\cat{G}_5(\str{C}_5,\str{C}_5)$}; 


\node (0) at (-1,6)  {$ \emptyset$};

 \node [xshift=-.1cm,yshift=.1cm] at (-1,6) {} edge [in=160,out=80,kloop] (v0z0k); 
 \node [xshift=.2cm] at (0-1,6) {} edge [in=60,out=120, ggloop] (v0z0gg); 
  \node [xshift=-.1cm,yshift=-.1cm] at (-1,6) {} edge [in=-120,out=-60, gloop, thin] (v0z0g);

\node (123) at  (-1,0) {$123$}  ;
\node (12) at (-3,2) {$12$};
 \node [xshift=-.1cm,yshift=.1cm] at (-3,2) {} edge [in=160,out=80, kloop] (12k); 
\node (13) at (-1,2) {$13$};

\node (23) at (1,2) {$23$};

\node (1) at (-3,4) {$1$};
 \node [xshift=-.1cm,yshift=.1cm] at (-3,4) {} edge [in=160,out=80, kloop] (1k); 
 \node [xshift=.2cm] at (-3,4) {} edge [in=60,out=120, ggloop] (1gg); 
\node (2) at (-1,4) {$2$};
 \node [xshift=-.1cm,yshift=.1cm] at (-1,4) {} edge [in=160,out=80, kloop] (2k); 

\node (3) at (1,4) {$3$};
   \node [xshift=-.1cm,yshift=-.1cm] at (1,4) {} edge [in=-120,out=-60, gloop] (3g); 

\karr{(3)--(0)}
\karr{(13)--(1)}
\karr{(23)--(2)}
\karr{(123)--(12)}

\ggarr{(3)--(2)}
\ggarr{(13)--(12)}
\garr{(2)--(1)}
\garr{(23)--(13)}

\end{tikzpicture}
\caption{Translation applied to the algebra $\str{C}_5$ in $\cat G_5 $}\label{fig:C5}
\end{figure}

Of course we have  a special situation here because $\str{C}_5$ is a chain.  
In general each layer of the dual space will not be a single  $\newapprox$-equivalence
class.  

We elected here to use the endomorphism $h_3$ in the alter ego, 
rather than the alternatives $h_1$ and $h_2$ supplied by 
Theorem~\ref{thm:newduality}, because the action of~$h_3$ on 
$\End\str{C}_5$ is especially simple.  
However one feature of the translation is present whichever of 
$h_1$, $h_2$ and $h_3$ we include  the alter ego:  ${\newapprox}$ 
is determined solely by $g_1$ and $g_2$,
whereas the ordering amongst ${\newapprox}$-equivalence classes 
is determined solely by $h_i$, whichever value of~$i$ we choose. 
  
We should draw attention, however, to Theorem~\ref{Theo:Fullness} below
in which we show that in any application in which we need a full 
(or equivalently a strong) duality for $\Gn$, then we must use $h_1$ rather 
 than any other~$h_{i}$.
\end{ex}

We  now seek to demonstrate that the process for passing from 
$\D(\str{C}_n)$ to $\fnt{HU}(\str{C}_n)$ (for $n \geq 4$) 
is much less transparent using
a duality based solely on endomorphisms 
(as in \cite{D76,DT05})
than when we use any of the variants
supplied by Theorem~\ref{thm:newduality}.

\begin{ex}  
We shall consider the alter ego  
$(C_5; h_1,h_2,h_3,\Tp)$ for $\G_5$.
As shown by   Davey and Talukder \cite[Theorem~2.4]{DT05}, this yields an optimal
duality.
In Fig.~\ref{fig:DTalterego}, the action  of $h_1$, $h_2$ and~$h_3$ on $\End \Cn$ is shown by  dashed,  
dotted and solid arrows, respectively. 
 It can be seen that these maps do encode
${\newapprox}$ and that, on the associated quotient, we can recover
the ordering of $\fnt{HU}(\str{C}_5)$.  
What does emerge clearly from this example is that 
translation from an endomorphism-based  duality to Priestley/Esakia duality
can be quite complicated, even on very simple objects.
Moreover fully reconciling our approach with that in \cite{D76} is not 
a trivial exercise in practice, though the theory ensures that it is, of course,
possible.

 \begin{figure} [ht]
\hspace*{-1.8cm}   
\begin{tikzpicture}[scale=1,
text depth=0.25ex,
 move up/.style=   {transform canvas={yshift=1pt}},
 move down/.style= {transform canvas={yshift=-1.2pt}},
 move left/.style= {transform canvas={xshift=-2pt}},
 move right/.style={transform canvas={xshift=2.5pt}}] 


\node (0) at (-1,6)  {$\emptyset$};

 \node [xshift=-.2cm,yshift=.1cm] at (-1,6) {} edge [in=160,out=80, kloop] (v0z0k); 
 \node [xshift=.4cm] at (0-1,6) {} edge [in=60,out=120, ggloop] (v0z0gg); 
  \node [xshift=-.2cm,yshift=-.1cm] at (-1,6) {} edge [in=-120,out=-60, gloop] (v0z0g);

\node (123) at  (-1,0) {$123$}  ;
\node (12) at (-3,2) {$12$};
 \node [xshift=-.1cm,yshift=.1cm] at (-3,2) {} edge [in=160,out=80, kloop] (12k); 
\node (13) at (-1,2) {$13$};

\node (23) at (1,2) {$23$};

\node (1) at (-3,4) {$1$};
 \node [xshift=-.1cm,yshift=.1cm] at (-3,4) {} edge [in=160,out=80, kloop] (1k); 
 \node [xshift=-.1cm,yshift=-.1cm] at (-3,4) {} edge [in=-120,out=-60, gloop] (1g); 
\node (2) at (-1,4) {$2$};
 \node [xshift=-.1cm,yshift=.1cm] at (-1,4) {} edge [in=160,out=80, kloop] (2k); 

\node (3) at (1,4) {$3$};
\node (3d) at (1,3.9) {};
\node (2d) at (-1,3.9) {};
\node (3u) at (1,3.95) {};
\node (2u) at (-1,3.95) {};
\harr{(123)--(12)}
\harr{(13)--(1)}
\harr{(23)--(2)}
\harr{(3)--(0)}
\draw [-stealth, thick,dotted,shorten >=.3pt,looseness=.8]  (3) to [bend right] (0);
\draw [-stealth,thick,  dashed,shorten >=.3pt,looseness=.8]  (3) to [bend left] (0);
\draw [move down,-stealth, thick,  dashed, shorten >=.3pt,
looseness=.3]  (23) to [bend left] (3);
\draw [move down,move left,-stealth,thick, dotted, shorten >=.3pt,looseness=.4]  (23) to [bend right] (3);
\draw [move down,-stealth, thick, dotted, shorten >=.3pt,looseness=.4]  (13) to [bend right] (1);
\draw [move down, -stealth, thick, dashed, 
looseness=.4]  (2) to [bend right] (3);


\draw [-stealth,thick,  dashed, shorten >=.3pt,looseness=.8]  (12) to [bend left] (23);

\earr{(123)--(13)}
\earr{(12)--(13)}

\eearr{(123)--(23)}
\eearr{(13)--(2)}
\eearr{(1)--(2)}
\earr{(2)--(3)};  


\end{tikzpicture}
\caption{The dual of $\str{C}_5$ for the duality for $\G_5$  with alter ego $(C_5; h_1,h_2,h_3, \Tp)$ \label{fig:DTalterego}}
\end{figure}
\end{ex}

Theorem~\ref{thm:going back}
presents the other half of the two-way 
translation process between $\CFn$ and $\cat Y^{\sigma }$, for a 
given $\sigma \in \Sigma_n$, as it applies to 
an object 
$\fnt{HU}(\A)$, 
where $\A \in \Gn$.
Lemma~\ref{lem:chain-hom} sets up, for the  given algebra $\A$,
mutually inverse bijections  $\iota_{\A}$ and $\gamma_{\A}$ between 
$\Gn(\A,\Cn)$ and a specified subset of  $\fnt{HU}(\A) \times \powersetn$.  
Starting  from the Esakia space
$\fnt{HU}(\A)$, we form suitable pairs with elements of $\powersetn$ and  
$\fnt{HU}(\A)$.
We then form  a topological structure $\X$ by equipping  our  set of 
pairs  with a topology and  operations 
$\sigma_{i}^{\XX}$ and 
establish that $\X$ is isomorphic to the natural dual space 
$\D^\sigma(\A)$. 
 We  carry out this construction   
using only the order and topological structure of  $\fnt{HU}(\A)$
(see the remarks following the theorem for the significance of this).
The Priestley/Esakia duality applied to $\A \in \Gn$ tells us that 
the evaluation map 
$k_{\A}  
 \colon\A \to  \fnt{KHU}(\A) $ is an 
isomorphism and the sets of the form $k_{\A}(a)$, as~$a$ ranges 
over~$\A$, are precisely the clopen up-sets in $\fnt{HU}(\A)$. 
 Moreover,  we recall that $\D(\A)$ is topologised with the subspace 
topology induced by the product topology  on $(C_n)^A$,  where 
the topology on $C_n$ is discrete. 
  These observations underlie the way topology 
is handled in the theorem.

\begin{theorem}\label{thm:going back}
Let $\A$ be an algebra in $\Gn$ and let $\fnt{HU}(\A) = (Y; \leq,\Tp)$.  
Let
\[
X= \{\, (y,U) \in  Y\times\powersetn 
\mid  |U| = d(y)+1 \mbox{ and\,  }0,n-1\in U\,\}.
\]
Define  partial maps $g_i^{\XX} $ and $f_i^{\XX}$ on $X$ as follows, 
where the  domains are given by the indicated  restrictions:
\begin{align*}
g_i^{\XX} (y,U) & =
          (y,g_i(U)) \qquad \ \ \text{\rm if } i\notin U, \\
f_i^{\XX} (y,U) & = 
              (y,f_i(U)) \qquad \ \ \text{\rm if } i+1\notin U, \\
\intertext{and total maps $h_j^{\XX}$ given by}
h_j^{\XX} (y,U)  & = \begin{cases}
                    (y,h_j(U))  & \text{\rm if } n-2\notin U, \\
                  (s(y),h_j(U)) & \text{\rm otherwise};
              \end{cases} 
\end{align*}
here $s$ denotes the function which, 
on 
$\fnt{HU}(\A)$---a  
forest of 
finite depth---maps each non-maximal point to its unique upper 
cover and fixes each maximal point.

For each  clopen up-set $W$ of $\fnt{HU}(\A)$ and  each $i\in\Cn$,  
let 
\[
A_{W,i}=
\{\,(y,U)\in X\mid i\in U\mbox{ and }|{\downarrow }i\cap U|
=|{\uparrow} x\cap W|+1 \,\}.
\]
 Let $\sigma \in\Sigma_n$. 
Then $\D^{\sigma }(\A)\cong \X=(X; \sigma _{1}^{\XX},\sigma _1^{\XX}, \ldots,\sigma _{n-2}^{\XX},\Tp^{\XX})$, 
 where 
$\Tp^{\XX}$ is the topology generated by the family of sets of the form $A_{W,i}$.
\end{theorem}

\begin{proof}
Lemma~\ref{lem:chain-hom} sets 
up  mutually inverse  bijections
$\iota_{\A} \colon \Gn(\A,\Cn) \to X$ and  
$\gamma_{\A}  \colon X\to \Gn(\A,\Cn)$, where 
$\iota_{\A}  \colon x \mapsto (\w \circ x, \ran x)$. 
From the description of $\gamma_{\A} $ given there, for each 
$a\in\A$ and $i\in\Cn$, we have  $(\gamma_{\A} (y,U))(a)=i$ 
if and only if $i\in U$ and 
$|{\downarrow} i\cap U|-1=
|\{u\in {\uparrow}y\mid u\in k_{\A}(a)\}|$.
Thus 
\[
\gamma_{\A} (A_{k_{\A}(a),i})=\{\,x\in \D(\A)\mid x(a)=i\,\}
\]
and hence $\iota_{\A} $ determines a homeomorphism between $(X;\Tp^{\X})$ and~$\D^\sigma(\A)$.

   In what follows $g_i^{\scriptscriptstyle\D(\A)}$, $f_i^{\scriptscriptstyle\D(\A)}$ and $h_j^{\scriptscriptstyle\D(\A)}$ denote the lifting of the (partial) maps $g_i$, $f_i$ and $h_j$ to $\Gn(\A,\Cn)$ for each $i\in\{1,\ldots,n-3\}$ and $j\in\{1,\ldots,n-2\}$.

Let $i\in\{1,\ldots, n-3\}$ and $x\in\Gn(\A,\Cn)$. 
Then $x\in \dom (g_i^{\scriptscriptstyle\D(\A)})$ if and only if 
$\ran x\subseteq \dom g_i$, that is, if $i\notin \ran x$. 
In this case
\begin{align*}
\iota_{\A} (g_i^{\scriptscriptstyle \D(\A)}(x))&=\iota_{\A} (g_i\circ x)=(\w\circ g_i\circ x, \ran( g_i\circ x))\\
&=(\w\circ f, g_i(\ran x))=g_i^{\XX}(\iota_{\A} (x)).
\end{align*}
The proof that $\iota_{\A} \circ f_i^{\scriptscriptstyle \D(\A)}=
f_i^{\XX}\circ \iota_{\A} $
is similar.
Finally, let $j\in\{1,\ldots, n-2\}$ and $x\in\Gn(\A,\Cn)$. 
Then 
\[
\iota_{\A} (h_j^{\scriptscriptstyle \D(\A)}(x))=\iota_{\A} (h_j\circ x)=(\w\circ h_j\circ x, h_j(\ran x)).
\]
Here we have two cases.
 If $n-2\notin\ran x$ then $\w\circ h_j\circ x=\w\circ x$. 
If $n-2\in\ran x$ then $\w\circ h_j\circ x=s(\w\circ x)$. 
We deduce that $\iota_\A$ preserves  $\sigma_i$ for 
$1 \leq i \leq n-2$.
\end{proof}

A remark is in order here on what we have really
achieved in Theorem~\ref{thm:going back}.  We have already 
observed that our `going back' construction builds 
(up to isomorphism)
 $\D^\sigma (\A)\in \cat Y^\sigma$ solely from the topology and order of the 
Esakia space $\fnt{HU}(\A)$
(as encoded by the map $s$).
  That is, the construction is performed 
without directly involving the algebra~$\A$.  This means that we 
can  carry out the procedure on any Esakia space in $\cat F_n$, 
regardless of whether or not the space is explicitly represented in 
the form $\fnt{HU}(\A)$ (as it can be, certainly).

We introduced the category $\cat{Y}^\sigma$ earlier but did not 
then give an explicit description of the equivalence between this 
category and $\cat F_n$
indicated in Fig.~\ref{Fig:functors}.  We can now remedy this 
omission.  
Define  $\fnt{F}^\sigma \colon \cat{Y}^\sigma \to \cat F_n$ on 
objects by letting 
\[
\fnt{F}^\sigma (\X) =(X/{\newapprox_{\XX}^{\sigma }};\leq_{\XX}^{\sigma },\Tp^{\XX}/{\newapprox_{\XX}^{\sigma }}).\]
Now define 
 $\fnt{F}^{\sigma }(\eta)$, for each  $\eta\in\cat{Y}^{\sigma }(\X,\Y)$,
 as follows:
\[
\fnt{F}^{\sigma }(\eta)([x]_{\newapprox_{\XX}^{\sigma }})=[\eta(x)]_{\newapprox_{\YY}^{\sigma }}  \quad \text{(for  $x\in \X$)}.
\]
The fact that $\fnt{F}^{\sigma }(\eta)$ is well defined follows from the definition of $\newapprox_{\YY}^{\sigma }$ and the fact that $\eta$ preserves $\sigma _1,\ldots,\sigma _{n-3}$. Since $\eta$  preserves $\sigma_{n-2}$, it is straightforward to check that $\fnt{F}^{\sigma }(\eta)$ is an Esakia morphism. 
Then
$\fnt{F}^{\sigma }$ is a functor naturally equivalent to
 $\fnt{HU}\restrict{\Gn} \,
 \fnt{E}^{\sigma }\restrict{\cat Y^{\sigma }}$. 

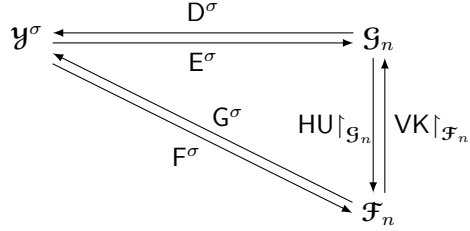
\begin{figure}  [ht]
\begin{center}
\begin{tikzpicture} 
[auto,
 text depth=0.25ex,
 move up/.style=   {transform canvas={yshift=1.9pt}},
 move down/.style= {transform canvas={yshift=-1.9pt}},
 move left/.style= {transform canvas={xshift=-2.2pt}},
 move right/.style={transform canvas={xshift=2.2pt}}] 
\matrix[row sep= 1.8cm, column sep= 4cm]
{ 
\node (Yn){$\cat Y^{\sigma }$};& \node (Gn) {$\Gn$};\\
&\node (CFn) {$\CFn$}; \\
};
\draw [latex-,move left ] (CFn) to node [xshift=4pt]  {$\fnt{HU}\restrict{\Gn}$}(Gn);
\draw [-latex,move right] (CFn) to node [swap]  {$\fnt{VK}\restrict{\CFn}$}(Gn);
\draw [-latex,move up] (Gn) to node [swap] {$\D^{\sigma }$}(Yn);
\draw [latex-,move down] (Gn) to node {$\E^{\sigma }$}(Yn);
\draw [latex-, transform canvas={yshift=-1pt,xshift=0pt}](Yn) 
to 
node [yshift=-4pt]  {$\fnt{G}^{\sigma }$} (CFn);
\draw [-latex, transform canvas={yshift=-5pt,xshift=0pt}](Yn) 
to 
node [swap,xshift=3pt] {$\fnt{F}^{\sigma }$} (CFn);
\end{tikzpicture}
\end{center}
\caption{The equivalence set up by the functors $\fnt{F}^{\sigma }$ 
and $\fnt{G}^\sigma$}  
 \label{Fig:functors}
\end{figure}

As with $\fnt{F}^{\sigma }$, the assignment 
$\Y\mapsto \fnt{G}^{\sigma }(\Y)= 
(X; \sigma _{1}^{\XX}, \ldots,\sigma _{n-1}^{\XX},\Tp^{\XX})$ 
can be extended to a functor from $\CFn$ into $\cat{Y}^{\sigma }$. 
In this case we will not present explicitly the action of 
$\fnt{G}^{\sigma }$ on maps. 
For our purposes, it is enough to observe that 
Theorems~\ref{thm:RevEngGodel} and \ref{thm:going back} imply 
that $\fnt{G}^{\sigma }$ can be made into a functor naturally 
equivalent to 
$\D^{\sigma} 
\fnt{V}\, 
\fnt{K}\restrict{\CFn}$; 
equivalently, $\fnt{F}^{\sigma }$ together with $\fnt{G}^{\sigma }$ 
determine a categorical equivalence between $\cat Y^{\sigma }$ 
and $\CFn$.  

In the same way as we did in  Section~\ref{sec:natdual} we end this section with a comment about $\G_3$ and $\G_2$. 
Using the alter ego $\twiddle{\C_3}=(\C_3;h_1,\Tp)$ for $\C_3$, 
Davey  \cite[pp.~126--127]{D76} 
shows how 
to obtain  $\fnt H(\A)$ from $\D(\A)$
 he first observes that  $\fnt H(\A)$ and $\D(\A)$ are homeomorphic as topological spaces. In our terms this means that $x\newapprox y$ if and only if $x=y$. This is actually the natural way to define $\newapprox$ in the absence of partial endomorphisms. Moreover, if we identify $\D(\A)/{\newapprox}$ with $\D(\A)$ 
then  
the order Davey defines on $\D(\A)$ is exactly the reflexive (transitive) closure of $\newcover$. 
This is only one side of the translation; 
the other direction can be obtained by the same construction and argument used in Theorem~\ref{thm:going back}. 
For $\G_2$,
 the translation between the  duality yielded by $(C_2;\Tp)$ and the Priestley/Esakia duality
is essentially that whereby a Boolean space is regarded as a special case of a Priestley space.

\section{The quest for full dualities} \label{Sec:Fullness}

Our first application of the translations developed in 
Section~\ref{sec:Translation} is to pick out  the full dualities from 
among the dualities we developed in Theorem~\ref{thm:newduality}.

In Theorem~\ref{Theo:Fullness} 
we show that, for $n \geq 4$, all  choices of $\sigma\in \Sigma_n$
 except $\sigma =(g_1,\ldots,g_{n-3},h_1)$ 
lead to  dualities which are not full.  Our
strategy is similar to one used by Davey; see his proof 
that $\End\Cn$ fails
to dualise $\Gn$ fully when $n \geq 4$ \cite[p.~127]{D76}.
     
The primary tool for establishing that a natural duality is full
is to establish that it is strong (see \cite[Chapter~3]{CD98} for the definitions and
discussion). The dualities for $\G_2$ and for $\G_3$ yielded  by the alter egos $ (C_2; \Tp)$ and 
$(C_3; h_1,\Tp)$ contain no partial 
operations. 
 They are known to be strong (see \cite[Theorem~4.2.3(ii)]{CD98}  
for the case $n=3$) and hence full.    
When $n=4$, the fact that 
$\twiddle{\C_n^\sigma}$, with $\sigma=(g_1,h_1)$, 
determines a (strong and hence) full duality for $\G_4$ was proved by Davey and Talukder in \cite[Theorem~6.1]{DT05}.
 We shall show that,   for any 
${n\geq 4}$, 
the dualising set  
$\{g_1, \ldots, g_{n-3}, h_1\} $ yields a full duality.  Our proof uses 
Theorems~\ref{thm:RevEngGodel} and~\ref{thm:going back}.  
 It is this technique, and  the fact that fullness is obtained directly, and not 
via strongness,  that we wish  to accentuate here.  
We note also that Davey's proof of fullness of his endomorphism-based duality for $\G_3$
\cite[pp.~126--127]{D76} may be seen as essentially a very special case of our method.  
Our  full dualities 
are necessarily strong;  see \cite[pp.~13--14]{DHT07}, where G\"odel algebras are called relative Stone Heyting algebras.  (The paper \cite{DHT07}, a stepping stone 
along the way to the final resolution in the negative of the longstanding ``ful
equals strong?" question,   identifies various well-known varieties for which 
non-strong full dualities cannot be found.)    
As an aside, we note that  
 a small generating set for 
the monoid $\PE\Cn$ is needed if axiomatisation of the dual category is to feasible.  
We can claim to have set up as good a full duality 
as is possible for addressing the axiomatisation problem for general~$n$.
However  in this paper we shall not seek an  extension of
\cite[Theorem~6.1]{DT05},  which relies both on  a suitable generating set for $\PE \C_4$  being found by hand and on standardness arguments
(see \cite[Section~3]{DT05}).

\begin{theorem}\label{Theo:Fullness} Let $n\geq 4$ and $\sigma\in \Sigma_n$.
Then
the alter ego  $\twiddle{\Cn^{\sigma }}$ fully dualises~$\Gn$ if and only if
 $\sigma =(g_1,\ldots,g_{n-3},h_1)$. 
\end{theorem}
\begin{proof}

Assume first that $\sigma \neq (g_1,\ldots,g_{n-3},h_1)$. 
We divide the problem into two cases. Both  proofs employ  the 
same tool. 
Since the functors 
$\fnt{G}^{\sigma }\colon \CFn\to \cat{Y}_n^{\sigma }$ and 
$\fnt{F}^{\sigma }\colon \cat{Y}_n^{\sigma }\to \CFn$ determine a 
categorical equivalence, $\X$ is isomorphic to 
$\fnt{G}^{\sigma }(\fnt{F}^{\sigma }(\X))$ for every 
$\X\in\cat{Y}_n^{\sigma }$.
Therefore we present a space $\X$ in 
$\IScP(\twiddle{\Cn^{\sigma }})$ and we observe that $\X$ is not 
isomorphic to $\fnt{G}^{\sigma }(\fnt{F}^{\sigma }(\X))$, which 
proves that 
$\cat{Y}_n^{\sigma }\neq \IScP(\twiddle{\Cn^{\sigma }})$, that is,   
$\twiddle{\Cn^{\sigma }}$ does not yield a full duality.
In the proof below we shall omit subscripts and superscripts,
for example from ${\newapprox},$ where these are clear from the 
context. 

\noindent{\bf Case 1}: Assume that
$\sigma_i=f_i$
for some $i\in\{1,\ldots,n-3\}$.

Let $j$ be such that $\sigma_j = f_j$ and $\sigma_i = g_i$ for any
 $j <i \leq n-3$.  
 Let $\X$ be  the subspace of $\twiddle{\Cn^{\sigma }}$ whose 
universe $X$ is $\{j+1,\ldots,n-1\}$. 
Observe that $\dom \sigma_i^{\XX}=X$ if $i\neq j$ and 
$\dom \sigma _j^{\XX}=\{j+2,\ldots, n-1\}$. 
Now assume $\X\in \cat{Y}^{\sigma }$. 
By assumption,
the quotient space  
$X/{\newapprox}$ has two elements 
$  \{j+1,\ldots,n-2\}$ and $\{n-1\}$. Then
$\fnt{F}^{\sigma }(\X)=
(\{\,[n-2]_\newapprox,[n-1]_\newapprox\,\};\leq,\Tp)$, where 
$\Tp$ is the discrete topology and 
$[n-2]_\newapprox<[n-1]_\newapprox$.  
The universe of  $\fnt{G}^{\sigma }(\fnt{F}^{\sigma }(\X))$ is 
${
\{([n-1]_\newapprox,\emptyset),([n-2]_\newapprox,\{1\}),\ldots,
([n-2]_\newapprox,\{n-2\})\}}$.
It follows that $|\fnt{G}^{\sigma }(\fnt{F}^{\sigma }(\X))|=n-1$ and 
$|X|=n-1-j$.
Therefore $\fnt{G}^{\sigma }(\fnt{F}^{\sigma }(\X))$  and $\X$ are 
not isomorphic.
 
\noindent{\bf Case 2}: Assume that $\sigma _i=g_i$ for 
$1 \leq i \leq n-3$ and  $\sigma_{n-2} =h_i$, where $i\neq 1$.  

With this assumption $h_i(1)=1$,  and so  $\{1,n-1\}$ is a closed 
subuniverse of $\twiddle{\Cn^{\sigma }}$. 
Let $\X$ be the subspace of $\twiddle{\Cn^{\sigma }}$ whose 
universe $X$ is $\{1,n-1\}$. 
Observe that $\dom \sigma_i^{\XX}=X$ if $i\neq 1$ and 
$\dom \sigma _1^{\XX}=\{n-1\}$. 
Now assume $\X\in \cat{Y}^{\sigma }$. 
Since $\sigma _i=g_i$ for $1 \leq i \leq n-3$, the definition of 
$\newapprox$ implies that $X/{\newapprox}$ has two classes 
$\{1\}$ and $\{n-1\}$. 
Because  $h_i(1)=1$ and $h_i(n-1)=n-1$, the space 
$\fnt{F}^{\sigma }(\X)$ is 
$(\{\,[1]_\newapprox,[n-1]_\newapprox\,\}; =,\Tp)$, where 
$\Tp$ is the discrete topology.  
The universe of $\Z=\fnt{G}^{\sigma }(\fnt{F}^{\sigma }(\X))$ is 
$\{([1]_\newapprox,\emptyset),([n-1]_\newapprox,\emptyset)\}$
and $\dom \sigma _1^{\ZZ}=Z$. 
Since $\dom \sigma _1^{\XX}=\{n-1\}$, the spaces $\Z$  and $\X$ 
are not isomorphic.

Now assume $\sigma =(g_1,\ldots,g_{n-3},h_1)$. 
Since $\twiddle{\Cn^\sigma }$ determines a duality on $\G_n$, for 
each non-empty set $S$,
 the space $(\twiddle{\Cn^\sigma })^{S}$ is isomorphic to the dual space of the 
$S$-generated free algebra $\Free_{\Gn}(S)$ in $\Gn$
\cite[Corollary 2.24]{CD98}.  
Therefore, to prove that  $\twiddle{\Cn^\sigma }$ determines a full 
duality, it is enough to prove that each closed substructure $\X$ of 
$\D^{\sigma }(\Free_{\Gn}(S))$, for some set $S$, is isomorphic to 
the dual space of some
 algebra in $\Gn$.  

Let us fix a non-empty set $S$ and a closed substructure $\X$ of 
$(\twiddle{\Cn^\sigma} )^{S}$.  
Let $x\in\X$ and $y\in (\twiddle{\Cn^\sigma })^{S}$ be such that 
$x\newapprox y$. We claim that $y\in \X$. 
By the definition of $\newapprox$ there is no loss of generality in 
assuming that $x=g_i^{\XX}(y)$ or $y=g_i^{\XX}(x)$ for some 
$i\in\{1,\ldots, n-3\}$. If $y=g_i^{\XX}(x)$, since $\X$ is a closed 
substructure of $(\twiddle{\Cn^\sigma })^{S}$, it follows directly 
that $y\in \X$. 
If $x=g_i^{\XX}(y)$ then, for each $s\in S$, we have $x_s=g_i(y_s)$
and hence $x_s\in\dom f_i$ and $f_i(x_s)=y_s$.   
Now observe that for $i\in\{1,\ldots,n-1\}$, the partial 
endomorphism $f_i$  equals 
$g_{i-1}\circ g_{i-2}\circ\cdots\circ g_2\circ g_1\circ h_1 \circ g_{n-3}\circ g_{n-2}\circ \cdots \circ g_{i+2}\circ g_{i+1}$. 
Then  
\[
y=
g_{i-1}^{\XX}\circ g_{i-2}^{\XX}\circ\cdots\circ g_2^{\XX}\circ 
g_1^{\XX}\circ h_1 ^{\XX}\circ g_{n-3}^{\XX}\circ g_{n-2}^{\XX}\circ 
\cdots \circ g_{i+2}^{\XX}\circ g_{i+1}^{\XX}(x),
\]
and we have established our claim. 

Let $\A = \Free_{\Gn}(S)$ and let 
$\Y=(\D^{\sigma }(\A)/{\newapprox};\leq,\Tp^{\YY})$ be
 as defined in Theorem~\ref{thm:RevEngGodel}.
The fact that, for each $x\in\X$,  the class $[x]_{\newapprox}$ is 
contained in $\X$ implies that 
$\X/{\newapprox}\subseteq \D^{\sigma }(\A)/{\newapprox}$. 
Let $\Z$ be the substructure of $\Y$ whose universe is $\X/{\newapprox}$. 
Since $\X$ is a closed subset of  $\D^{\sigma }( \A)$ and  
$\Tp^{\Y}$ is the quotient topology, $\X/{\newapprox}$ is a closed 
subset of $(\,\D^{\sigma }(\A) /{\newapprox};\Tp^{\Y}\,)$. 
From the fact that $\X$ is closed under $h_1$  and the definition of 
$\leq$ in $\D^{\sigma }(\A)/{\newapprox}$, it follows that  
$\X/{\newapprox}$ is an up-set in 
$(\D^{\sigma}(\A)/{\newapprox};\leq)$. 
We conclude that $\Z$ belongs to $\CFn$.
Therefore there exists  $\B\in\Gn$ such that 
$\fnt{HU}(\B)\cong \Z$. 
By Theorem~\ref{thm:going back}, 
$\D^{\sigma }(\B)\cong \fnt{G}^{\sigma }(\Z)$.

By Theorems~\ref{thm:RevEngGodel} and~\ref{thm:going back},
the map $\iota_{\A} \colon \D^{\sigma }(\A) 
\to \fnt{G}^{\sigma }\fnt{F}^{\sigma }(\D^{\sigma }(\A))$, defined 
by  
$\iota_{\A} (x)= (\omega \circ x,\ran x)$, 
sets up an isomorphism. 
Since for $x\in\X$ the class $[x]_{\newapprox}$ is contained in 
$\X$, the restriction of $\iota_{\A}$ to $\X$ determines a bijection 
between $\X$ and $\fnt{G}^{\sigma }(\Z)$. 
Moreover, $\iota_{\A}{\restriction}_{\XX} $ is an 
isomorphism between $\X$ and $\fnt{G}^{\sigma }(\Z)$
in the category $\IScP(\twiddle{\Cn^{\sigma }})$. 
Putting the components of the proof together we deduce that
$\X\cong \fnt{G}^{\sigma }(\Z)\cong \D^{\sigma }(\B)$.
\end{proof}

\section{Applications: amalgamation and coproducts} 
\label{sec:coprod}

Our second application of our interactive duality theory concerns 
amalgamation in varieties of G\"odel algebras.  
Our main result here is Theorem~\ref{thm:APthm}.
Along the way we expose the interrelation between satisfaction of 
the amalgamation property in a finitely generated quasivariety and 
the existence of an alter ego which is a total structure yielding a 
strong duality:   the result given in Lemma~\ref{Lem:TotalDualAmal} 
is not new, but we do present a simpler and more self-contained 
proof of it than that of  \cite[Lemma~5.3]{CD98}.

Given a class of algebras $\CA$, a \defn{V-formation} is a quintuple 
$(\A,\B,\C,f_B,f_C)$ where 
$\A,\B,\C \in \CA$,   and 
$f_B\in\CA(\A,\B)$ and 
$f_C\in\CA(\A,\C)$ are  injective homomorphisms. 
A $V$-formation $(\A,\B,\C,f_B,f_C)$ 
\defn{admits amalgamation} if there exist an 
algebra~${\str{D}\in\CA}$ and embeddings $h_B\colon\B\to\str{D}$ 
and $h_C\colon\C\to\str{D}$ such that 
$h_B\circ f_B=h_C\circ f_C$. 
The class $\CA$ has the \defn{amalgamation property} if every 
$V$-formation admits amalgamation. 

Let  $\CA$ be a quasivariety. It is well known that $\CA$ admits all 
colimits  and in particular pushouts, also known as fibred 
coproducts (see \cite[Chapter~III]{McL} for definitions and 
properties of colimit, pushout, and coproduct). 
Let $(\A,\B,\C,f_B,f_C)$ be a $V$-formation in $\CA$. 
Then $(\A,\B,\C,f_B,f_C)$ admits amalgamation if and only if the 
pushout maps $p_B\colon\B\to\B\coprod_{\A}\C$ and 
$p_C\colon\B\to\B\coprod_{\A}\C$ are embeddings.
With this in mind, we shall  focus below on the properties of 
pushouts, and particularly on when pushout maps are embeddings.

\begin{lemma}\label{Lem:TotalDualAmal}
Let $\CA=\ISP(\M)$ be the quasivariety generated by a finite 
algebra~$\M$ and assume that there is a total structure~$\MT$
which yields  a strong duality on~$\CA$.  
Then~$\CA$ has the amalgamation property.
\end{lemma}
\begin{proof}
Let $\D\colon\CA\to\IScP(\MT)$ and $\E\colon\IScP(\MT)\to\CA$ 
be the functors determined by $\MT$. 
By \cite[Theorem~6.1.2 and Exercise~6.1]{CD98}, under our 
assumptions a homomorphism~$h$ in $\CA$ is an embedding if 
and only if $\D(h)$ is surjective.

Let $(\A,\B,\C,f_B,f_C)$ be a $V$-formation in $\CA$. 
Then $\D(f_B)$ and $\D(f_C)$ are surjections.
 Let $\X$ be the fibred product 
$\D(\B)\overset{\D(f_B)}{\longrightarrow}\D(\A)
\overset{\D(f_C)}{\longleftarrow}
\D(\C)$.
That is, $\X$ is the subspace of $\D(\B)\times\D(\C)$ whose 
universe is
\[
X=\{\,(x,y)\in\D(\B)\times\D(\C)\mid \D(f_B)(x)=\D(f_C)(y)\,\}.
\]
Since $\D(f_B)$ and $\D(f_C)$ are surjective, 
 the projection maps ${\pi_{B}\colon X\to\D(\B)}$ and 
$\pi_{C}\colon X\to\D(\C)$ are also surjective.   
Since the duality is full, it follows that $\E(\X)$ is the pushout of 
$(\A,\B,\C,f_B,f_C)$ and the  pushout maps are
${\E(\pi_{B})\circ e_\B\colon \B\to \E(\X)}$ and 
$\E(\pi_{C})\circ e_\C\colon \C\to \E(\X)$. 
Since $\E(\pi_{B})\circ e_\B$ and $\E(\pi_{C})\circ e_\C$ are 
embeddings, the $V$-formation $(\A,\B,\C,f_B,f_C)$ 
admits amalgamation.
\end{proof}

As we mentioned in Section~\ref{sec:intro}, 
in \cite{Maks} Maksimova  proved that 
$\Gn$ fails to satisfy the amalgamation property 
if $n \geq 4$. Combining this with Lemma~\ref{Lem:TotalDualAmal}, we obtain the following.
\begin{corollary}
The variety $\G_n$,  with 
$n\geq 4$,
does not admit a total 
{\upshape(}single-sorted\,{\upshape)} strong duality.
\end{corollary}

Given a natural duality, the problem of describing which 
maps between dual structures correspond to embeddings between 
algebras is not as simple in general 
as it is when the duality is strong and 
based on a total structure.
However  for the G\"odel algebra varieties $\Gn$  
 the task of describing such 
maps is greatly facilitated by two features which work to our 
advantage.  
Firstly, we can call on  the two-way translation between our natural  
dualities and the Priestley/Esakia duality.
Secondly,  the latter duality is the restriction to Heyting algebras of 
Priestley duality,  which is 
strong and has a total structure as the alter ego.

\begin{lemma}\label{Lem:CarEmbed}
Let $\A,\B\in\G_n$ with $n \geq 2 $ and $f\colon\A\to\B$ be a 
homomorphism. 
For each $\sigma\in \Sigma_n$, the following statements are 
equivalent:
\begin{newlist}
\item[{\upshape (1)}] $f$ is an embedding;
\item[{\upshape (2)}] for each $x\in  \D^{\sigma}(\A)$, there exists 
$y\in\D^{\sigma}(\B)$ such that 
$\D^{\sigma}(f)(y)
\newapprox^{\sigma}_{\scriptscriptstyle \D^\sigma (\A)} x$.
\end{newlist}
\end{lemma}
\begin{proof}
The map $f$ is an embedding if and only if $\fnt{HU}(f)$ is surjective. 
Since the functor $\fnt{HU}$ is naturally isomorphic to 
$\fnt{F}^{\sigma}\D^{\sigma}$, the map $\fnt{HU}(f)$ is surjective if and only 
if $\fnt{F}^{\sigma}(\D^{\sigma}(f))$ is surjective. 
From the observations we made about  
the functor $\fnt{F}^{\sigma}$ in  Section~\ref{sec:Translation} we have
$\fnt{F}^{\sigma}(\D^{\sigma}(f))([z]_{\newapprox ^{\sigma}
_{\scriptscriptstyle \D^\sigma (\B)}})=[\D^{\sigma}(f)(z)]_{\newapprox ^{\sigma}
_{\scriptscriptstyle \D^\sigma (\A)}}$.
 Therefore $\fnt{F}^{\sigma}(\D^{\sigma}(f))$ is surjective if and only if (2) holds.
\end{proof}

We are now ready to prove the main result of this section.

\begin{theorem} \label{thm:APthm}
Let $(\A,\B,\C,f_B,f_C)$ be a $V$-formation in $\Gn$ with 
$n\geq 4$. 
Let $\sigma=(g_1,\ldots,g_{n-3},h_1)$. 
Then the  following statements are equivalent:
\begin{newlist}
\item[{\upshape (1)}] $(\A,\B,\C,f_B,f_C)$ admits amalgamation;
\item[{\upshape(2)}] the dual maps $\D^{\sigma}(f_B)$ and $\D^{\sigma}(f_C)$ satisfy the following conditions: 
\begin{newlist}
\item[{\upshape (a)}] for each $x\in\D^{\sigma}(\B)$, there exist 
$y\in\D^{\sigma}(\B)$ and $z\in\D^{\sigma}(\C)$ such that 
$x\newapprox_{\scriptscriptstyle \D^{\sigma}(\B)}^{\sigma} y$ and 
$\D^{\sigma}(f_{B})(y)=\D^{\sigma}(f_{B})(z)$; and
\item[{\upshape (b)}] for each $x'\in\D^{\sigma}(\C)$, there exist 
$y'\in\D^{\sigma}(\C)$ and $z'\in\D^{\sigma}(\B)$ such that 
$x'\newapprox_{\scriptscriptstyle \D^{\sigma}(\C)}^{\sigma} y'$ and 
$\D^{\sigma}(f_{C})(y')=\D^{\sigma}(f_{B})(z')$.
\end{newlist}
\end{newlist}
\end{theorem}
\begin{proof}
In this proof we only consider $\sigma=(g_1,\ldots,g_{n-3},h_1)$, therefore we shall omit the superscripts and write $\twiddle{\Cn}$, $\D$ and $\E$ instead of $\twiddle{\Cn^{\sigma}}$, $\D^{\sigma}$ and $\E^{\sigma}$, respectively.

As in Lemma~\ref{Lem:TotalDualAmal}, let  $\X$ be the subspace of 
$\D(\B)\times\D(\C)$ whose universe is 
$\{\,(x,y)\in\D(\B)\times\D(\C)\mid \D(f_B)(x)=\D(f_C)(y)\,\}$,
 that is, the fibred product 
$\D(\B)\overset{\D(f_B)}{\longrightarrow}\D(\A)
\overset{\D(f_C)}{\longleftarrow}\D(\C)$. 
By Theorem~\ref{Theo:Fullness}, the duality determined 
by~$\twiddle{\Cn}$ is full. 
Therefore $\E(\X)$ is the pushout of $(\A,\B,\C,f_B,f_C)$ with 
pushout maps 
$\E(\pi_{B})\circ e_\B\colon \B\to \E(\X)$ and 
$\E(\pi_{C})\circ e_\C\colon \C\to \E(\X)$. 
Now observe that $\pi_{B}$ satisfies condition (2) in
Lemma~\ref{Lem:CarEmbed} if and only if (a) holds. 
This proves that $\E(\pi_{B})$ is an embedding if and only if (a) 
holds. 
Similarly, $\E(\pi_{C})$ is an embedding if and only if (b) holds. 
\end{proof}

Our last application concerns coproducts of G\"odel algebras.
The coproduct of a family of algebras depends not only on the 
algebras
involved
  but also on the class in which we form the coproduct.
In particular, given  a set  $\cat{K}$  of algebras in $\Gn\subseteq\G$, their  
coproduct in $\G$ might be different from their coproduct in 
$\Gn$. 
Nonetheless, if the set $\cat K$ is finite, the coproduct formed in 
$\G$ coincides with the coproduct formed in $\G_m$,  for a 
suitably large $m$; 
in particular, it is enough to consider
$m=2+\sum_{\A \in \cat{K}} (k_{\A}-2)$,
 where  $k_{\A}$ is the minimal $k$ 
 such that  $\A\in\G_k$ for  all $\A\in\cat K$.
These observations give us access  
to certain coproducts in $\G$. In particular we can calculate 
coproducts in $\G$ of any finite set of  algebras belonging to~$\bigcup_{n}\Gn$.

Natural dualities, when available,  are a good  tool for  the study of 
coproducts: under such a  duality, the dual space of the coproduct 
of a family of algebras corresponds to the cartesian product of their 
dual spaces. 
This fact was already noted in \cite[Section~5]{D76}, where it
was applied to determine finitely generated free algebras in
$\Gn$ and in $\G$, with the calculations performed wholly within the 
natural duality setting. 
The main difficulty one encounters 
in attempting to use
 a natural duality to  
calculate the coproduct in $\Gn$ or in $\G$ of a family of algebras 
$\cat K\subseteq\G_n$ 
lies in 
finding effective 
descriptions of the dual 
spaces of the algebras $\A\in\cat K$ and then of the algebra dual  
to  the resulting cartesian product. 

On the other hand,  using Priestley/Esakia duality instead,
as D'Antona and Marra do in \cite{DM}, brings different challenges.   
While dual spaces may be easy to visualise, the duality functor  
({\it viz.}~the restriction of $\fnt{HU}$ to $\Gn$; see Fig.~\ref{Fig:GodelDual})
does not in general convert a coproduct in $\Gn$ to a cartesian product, so that the dual space of a coproduct is hard to describe. 
D'Antona and Marra proceed in the following way to describe 
the coproduct of  any finite family $\cat K$  of finite 
(equivalently, finitely generated) G\"odel algebras.
They first find the Priestley/Esakia dual of each 
algebra in $\cat K$.
Then, using the fact that product 
distributes over coproduct 
in the category of Esakia spaces (see~\cite[p.~391]{DW81},
 where it is proved that, in every variety of Heyting algebras,
coproduct distributes over product), the problem reduces to 
describing the product of two finite trees in a suitable category. 
In bare outline, the authors' general strategy 
to construct the product of two trees is: 
first to represent each tree by a family of 
ordered partitions; second, 
 to construct  another family of ordered 
partitions 
from the original ones 
by  suitably shuffling and merging these. 
Finally they must  
reverse the process to obtain a tree from  the resulting set of 
ordered partitions.

Our two-way translation allows us to work with a natural duality and 
Priestley/Esakia in combination,  and so gives a simpler approach 
than one based on either of these duality techniques  alone.  
We present our method as a 5-step procedure.      
We shall fix $\sigma \in \Sigma_n$ to be 
$(g_1,\ldots, g_{n-3},h_{n-2})$ because this gives particularly 
simple pictures in examples, but any other choice of~$\sigma$ 
would also be legitimate. 
To  calculate $\coprod_{\Gn}  \cat K$, where 
$\cat K\subseteq\G_n$: 
\begin{newlist}
\item[{\bf 1.}] For each $\A\in\cat K$, determine $\fnt{HU}(\A)$.
\item[{\bf 2.}]  Use  Theorem~\ref{thm:going back} to construct  
$\fnt{G}^{\sigma}\fnt{HU}(\A)$ from $\fnt{HU}(\A)$, for each ${\A\in\cat K}$.
\item[{\bf  3.}]   Form  the cartesian product 
$\X=\prod \{\,\fnt{G}^{\sigma}\fnt{HU}(\A)\mid \A\in\cat K\,\}$.
\item[{\bf  4.}] Determine $\newapprox_{\XX}^{\sigma}$, $\newcover_{\XX}^{\sigma}$ and $(\X/{\newapprox}_{\XX}^\sigma;\leq_{\XX}^{\sigma},\Tp^{\XX}/{\newapprox}_{\XX}^\sigma)$ as in Theorem~\ref{thm:RevEngGodel}.
\item[{\bf  5.}] 
 $\fnt{VK}(\X/{\newapprox}_{\XX}^{\sigma},\leq_{\XX}^{\sigma},\Tp^{\XX}/{\newapprox}_{\XX}^\sigma)$
gives $\coprod_{\Gn}  \cat K$.
\end{newlist}

\noindent 
Of course the procedure described above  does not constitute 
an algorithm unless restricted to a finite family of finite algebras.

Some comments on our strategy  as compared with that in 
\cite{DM} should be made.  
Our procedure allows us on the one hand to be flexible 
and to calculate  coproducts in any  subvariety $\G_{k}$ that 
contains the algebras, and on the other hand, most significantly, we 
replace the involved procedure of calculating shuffles and merges
simply by the computation of a  cartesian product of structured 
spaces. 
It might be seen as a disadvantage that we need to 
know at the outset
in which G\'odel subvariety $\Gn$ we need to 
work in order that
 the coproduct calculated in $\Gn$ coincides 
with that calculated  in $\G$.
But this is a minor issue: since, as we observed before, a suitable 
$n$ can  easily be found once the finite family of algebras is 
fixed.
From the perspective of the procedure developed in~\cite{DM}, one 
can similarly cut down to a finitely generated subvariety $\G_k$.  
However to do so one needs first to calculate the Priestley/Esakia 
dual of the coproduct in $\G$ and then to truncate the resulting 
forest to obtain a forest of  depth $k-2$.

\begin{figure}
\includegraphics[scale=.7, trim=100 250 100 60, clip]{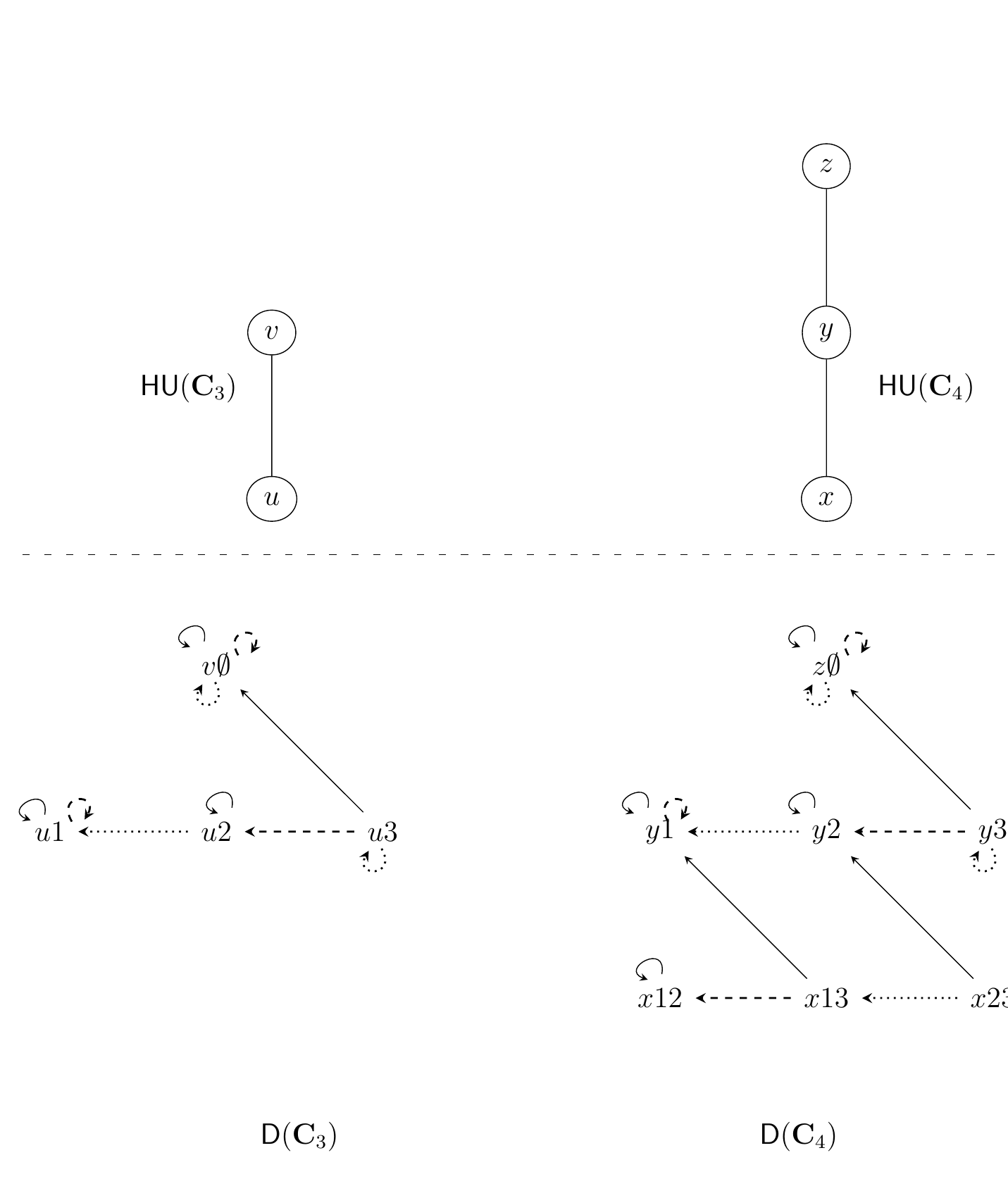}
\caption{Calculating $\str{C}_3\coprod_{\G_5} \str{C}_4$:  Steps 1 and 2} \label{Fig:C3C4G5}
\end{figure}

We now illustrate how our method works in practice.
The first example we choose is 
the one given in
\cite[Examples~1 and 2]{DM}.  
This is done to enable the reader   
to compare and contrast the two procedures.
(Observe that in \cite{DM} the order in the 
Priestley/Esakia spaces is the 
dual of the one considered in this 
paper.)

\begin{ex}  \label{ex:C3C4G5}
Let us determine  $\C_3 \coprod_{\G} \C_4$.
 First observe $\C_3\coprod_{\G}\C_4$ belongs to $\G_5$. 
We 
take 
$\sigma=(g_1,g_2,h_{3})$.

\noindent 
{\bf Step 1}: 
Let $\fnt{HU}(\C_3)=(\{u,v\};\leq)$ where $u< v$ and 
$\fnt{HU}(\C_3)=(\{x,y,z\};\leq)$ with $x<y<z$ 
(see~Fig.~\ref{Fig:C3C4G5}).  
\begin{figure}
\includegraphics[clip=true,angle=90,trim=90 340 35 118,clip]
{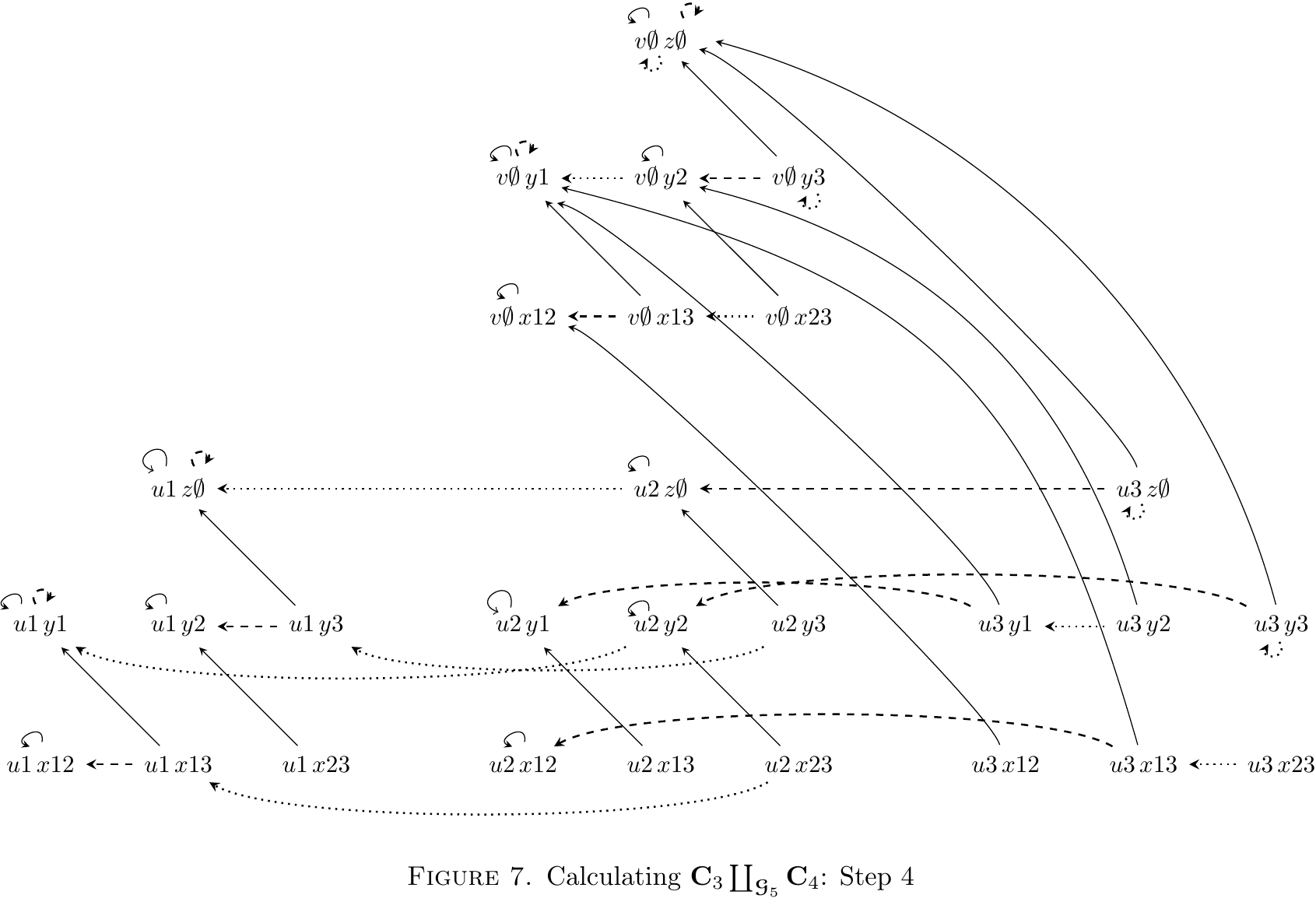} 
\vspace*{1cm}
\end{figure}

\setcounter{figure}{7}
\begin{figure}
\includegraphics[scale=.8,trim=140 420 80  100,clip  ]{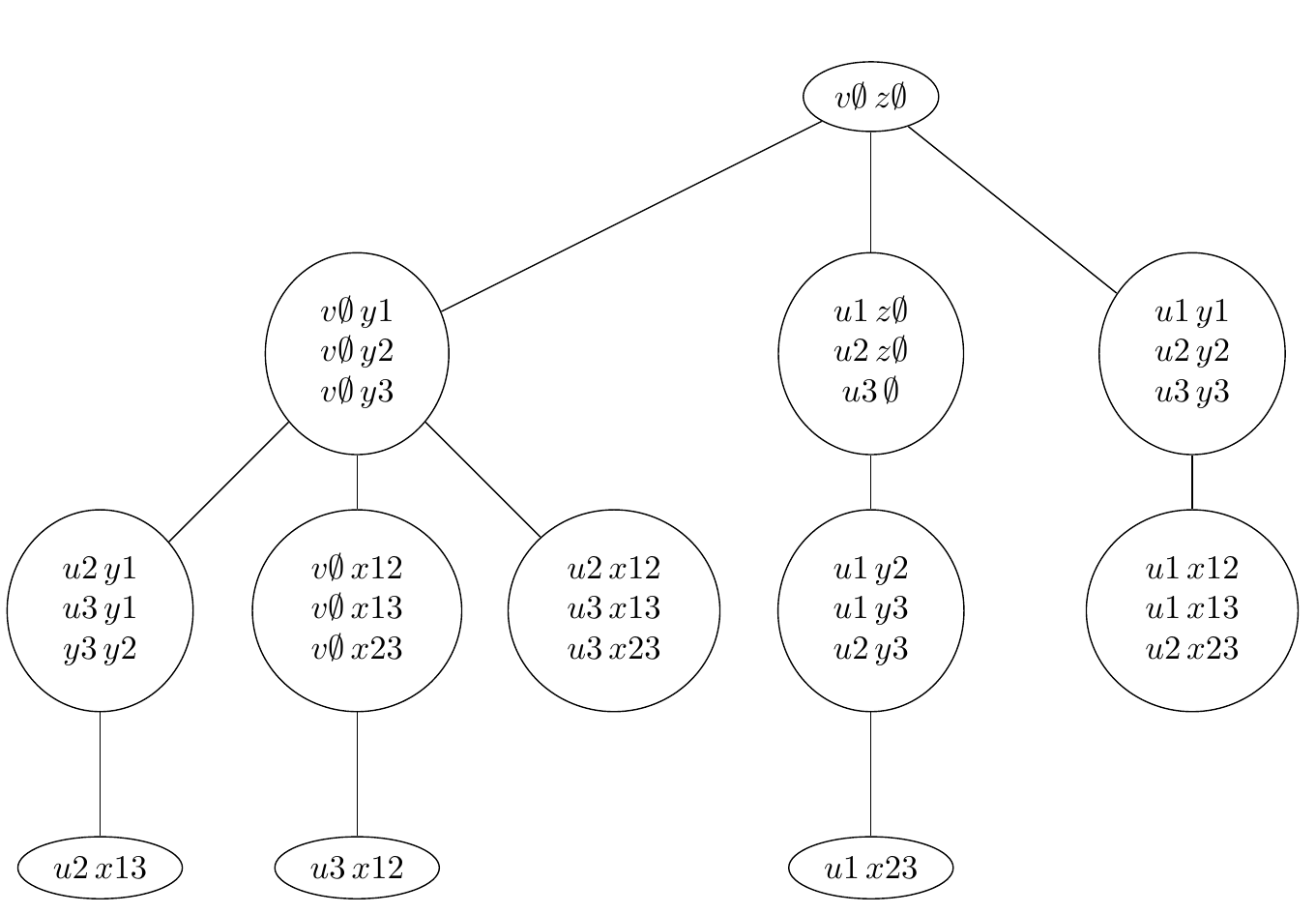}
\caption{Calculating $\str{C}_3\coprod_{\G_5} \str{C}_4$ :  Step 4
} \label{fig:cop3}
\end{figure}
\noindent {\bf Step 2}:  
Figure~\ref{Fig:C3C4G5} 
also  
shows $\fnt{G}^{\sigma}(\fnt{HU}(\C_3))$ and 
 $\fnt{G}^{\sigma}(\fnt{HU}(\C_4))$. 

\noindent
{\bf Step 3}: 
The cartesian product $\fnt{G}^{\sigma}(\fnt{HU}(\C_3))\times \fnt{G}^{\sigma}(\fnt{HU}(\C_3))$ is represented in 
Fig.~7. 

\noindent 
{\bf Step 4}:
  Figure~\ref{fig:cop3} depicts  
 the calculation of $\fnt{F}^{\sigma}(\fnt{G}^{\sigma}(\fnt{HU}(\C_3)\times\fnt{G}^{\sigma}(\fnt{HU}(\C_4))$ and therefore isomorphic to $\fnt{HU}(\C_3\coprod_{\G_5}\C_4)$.

\noindent 
{\bf Step 5}:
The  lattice whose Priestley dual  is shown in Fig.~\ref{fig:cop3} is
(the reduct of the  G\"odel algebra)  isomorphic to 
$
\bot \oplus\bigl((\bot\oplus( \C_3\times\C_3\times\C_2))\times \C_4\times \C_3\bigr)$,
where 
$\oplus$ denotes linear sum.  
\end{ex}

\begin{figure}
\includegraphics[scale=.9,trim=150 450 120  100,clip]{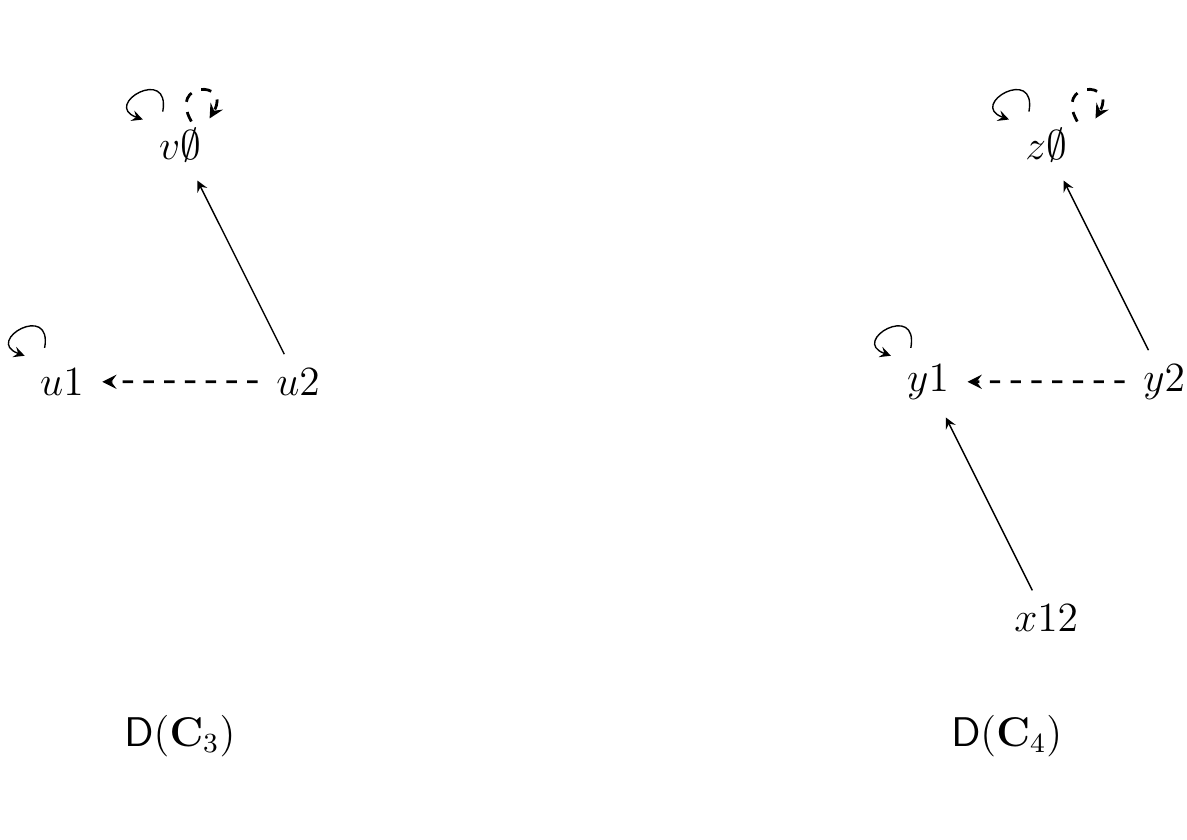}
\caption{Calculating $\str{C}_3\coprod_{\G_4} \str{C}_4$:  Step 2
} \label{fig:cop32}
\end{figure}

\begin{figure}
\includegraphics[scale=1,  
trim=170 400 140 100,clip]{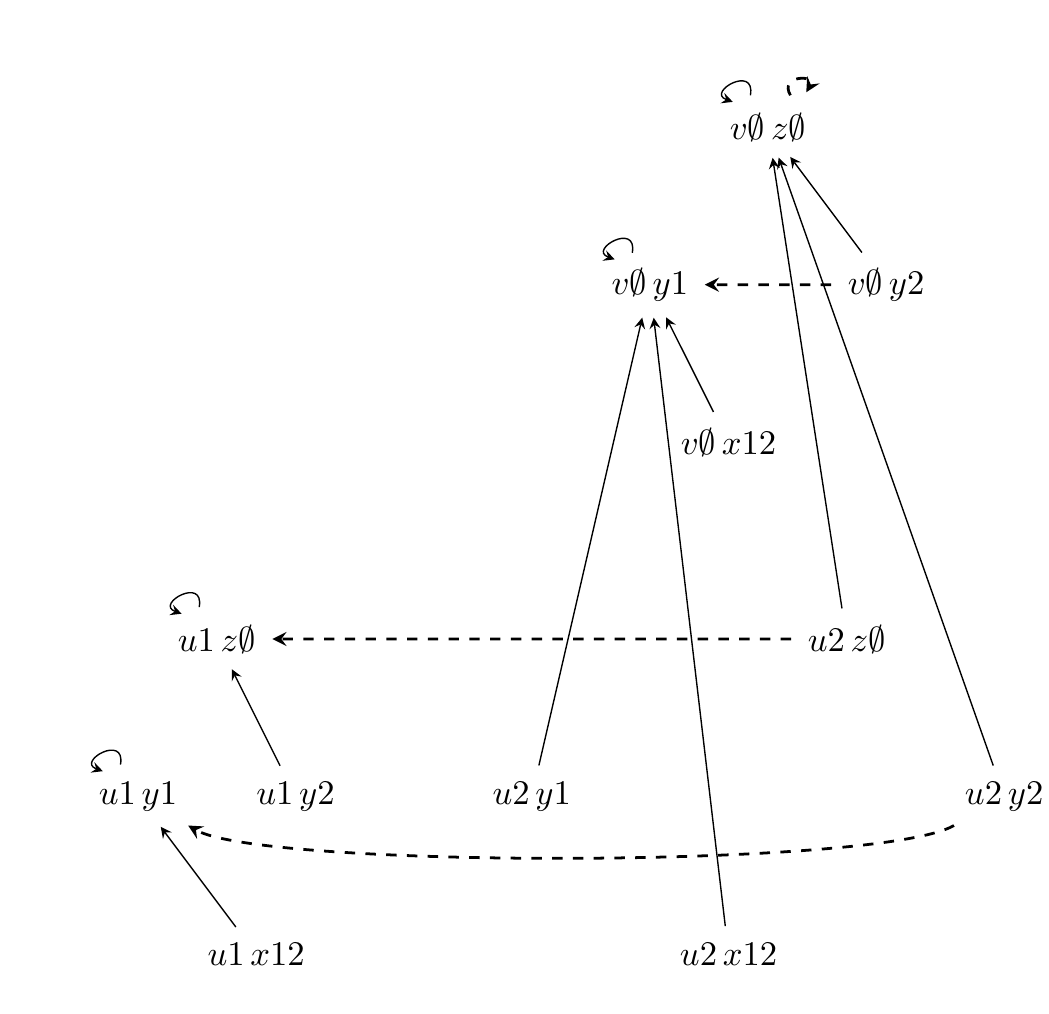}
\caption{Calculating $\str{C}_3\coprod_{\G_4} \str{C}_4$:  Step 3
} \label{fig:cop34}
\end{figure}

\begin{figure}
\includegraphics[scale=.8,trim=140 490  80   100,clip]{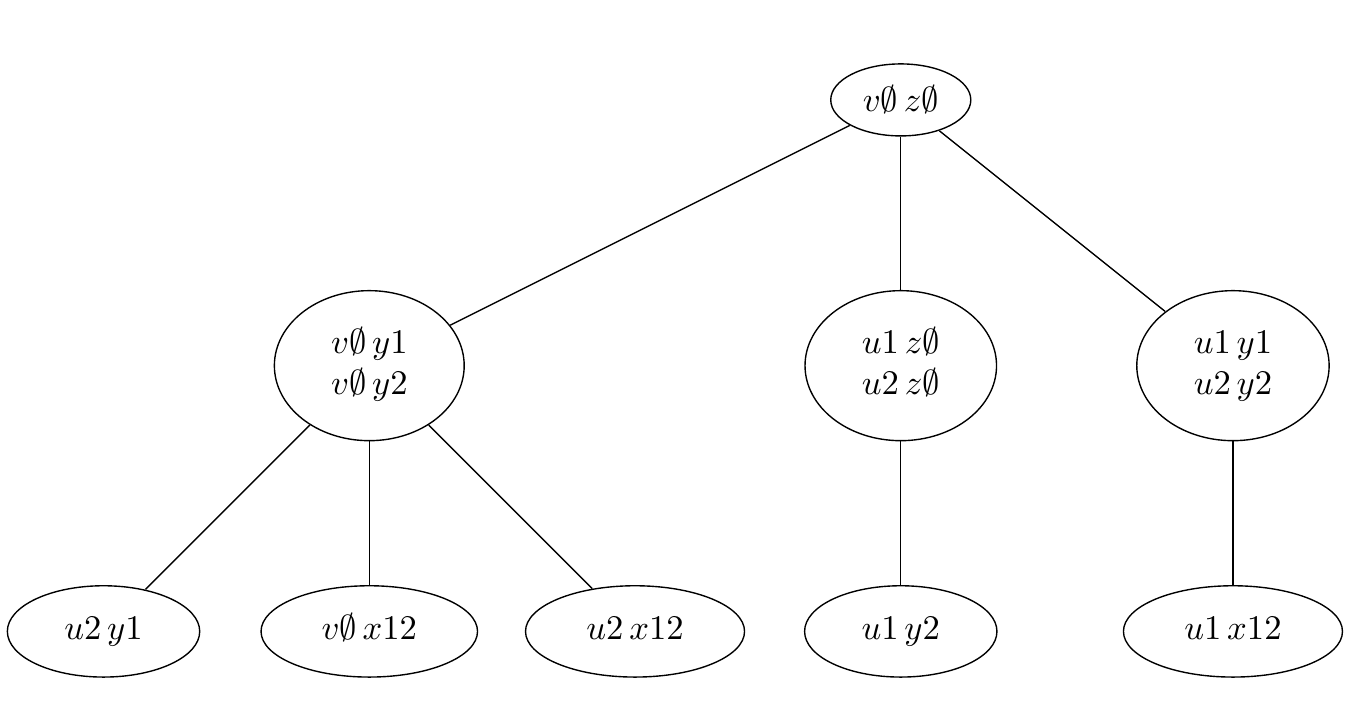}
\caption{Calculating $\str{C}_3\coprod_{\G_4} \str{C}_4$:  Step 4
} \label{fig:cop44}
\end{figure}

\begin{ex}
Since  $\C_3$ and $\C_4$ lie in $\G_4$, we can  also calculate 
 $\C_3\coprod_{\G_4}\C_4$. 
Here Step 1 is the same as in Example~\ref{ex:C3C4G5}.  Steps~2--4 are
shown in \hbox{Figs.~\ref{fig:cop32}--\ref{fig:cop44}}.  
We see that the tree
 $\X=\fnt{G}^{\sigma}(\fnt{F}^{\sigma}\fnt{HU}(\C_3)\times 
\fnt{F}^{\sigma}\fnt{HU} (\C_4))$ 
is (isomorphic to) a truncation of the one
obtained in Step~4 of Example~\ref{ex:C3C4G5}.
It is only for $n \geq 5$ that 
the coproduct in $\G_n$ coincides with that in $\G$.
Finally, to complete Step~5 we observe that the dual lattice of $\X$ is 
\[
\bot \oplus((\bot\oplus( \C_2\times\C_2\times\C_2))\times \C_3\times \C_3),
\]
which has 82 elements.
\end{ex}

\end{document}